\documentclass[11pt,a4paper]{amsart}

\usepackage{graphicx}
\usepackage{enumerate}
\usepackage{bbm}
\usepackage{amsmath, amssymb, amsthm, amscd}
\usepackage{mathtools}
\usepackage[dvipsnames]{xcolor}
\usepackage{caption}
\usepackage{subcaption}
\usepackage{multirow}
\usepackage{extarrows}
\usepackage{stmaryrd}
\usepackage{tikz-cd}
\usepackage[pdfencoding=auto, psdextra, hidelinks]{hyperref}
\usepackage{cleveref}
\usepackage{geometry}  
\usepackage{orcidlink}

\theoremstyle{plain}
\newtheorem{theorem}{Theorem}[section]
\newtheorem{lemma}[theorem]{Lemma}
\newtheorem{corollary}[theorem]{Corollary}
\newtheorem{proposition}[theorem]{Proposition}
\newtheorem{hypothesis}[theorem]{Hypothesis}

\theoremstyle{definition}
\newtheorem{definition}[theorem]{Definition}

\theoremstyle{remark}

\numberwithin{equation}{section}

\def\softd{{\leavevmode\setbox1=\hbox{d}%
              \hbox to 1.05\wd1{d\kern-0.4ex{\char039}\hss}}}

\def\bold#1{\mbox{\boldmath $#1$}}
\newcommand{\uu}[1]{\bold{#1}}

\newcommand{\ldblbrace}{\{\mskip-5mu\{}
\newcommand{\rdblbrace}{\}\mskip-5mu\}}
\newcommand{\diff}[1]{\llbracket #1\rrbracket}
\newcommand{\avg}[1]{\ldblbrace #1\rdblbrace}

\newcommand{\modL}[1]{\left|#1\right|}
\newcommand{\paraL}[1]{\left(#1\right)}
\newcommand{\sumintK}[1]{\sum_{\sigma \in \mathcal{E}(K)} \frac{|\sigma|}{|K|}}
\newcommand{\sumintKK}[1]{\frac{1}{|K|} \sum_{\sigma \in \mathcal{E}(K)} }
\newcommand{\sumK}[1]{\sum_{K \in \mathcal{T}} |K|}
\newcommand{\sumintall}[1]{\sum_{\sigma \in \mathcal{E}} |\sigma|}
\newcommand{\sumintalla}[1]{\sum_{\sigma \in \mathcal{E}} }

\newcommand{\R}{\mathbb{R}}

\newcommand{\tdom}{(0,T)}
\newcommand{\odom}{\tdom\times\Omega}

\newcommand{\bu}{{\uu{u}}}

\newcommand{\du}{\delta\bu}
\newcommand{\tvrho}{\widetilde{\vrho}}
\newcommand{\tu}{\widetilde{\bu}}
\newcommand{\tv}{\widetilde{\bfv}}
\newcommand{\tpi}{\widetilde{\pi}}
\newcommand{\D}{\mathcal{D}} 
\newcommand{\E}{\mathcal{E}}

\newcommand{\veps}{\varepsilon}

\newcommand{\vrho}{\varrho}
\newcommand{\vth}{\vartheta}
\newcommand{\uphi}{\uu{\varphi}}
\newcommand{\epso}{{0,\varepsilon}}
\newcommand{\half}{\frac{1}{2}}
\newcommand{\dx}{\mathrm{d}x}
\newcommand{\dbx}{\mathrm{d}\uu{x}}

\newcommand{\dt}{\mathrm{d}t}

\newcommand{\lint}{\int\limits_0^\tau\!\!\!\int\limits_\Omega}

\newcommand{\T}{\mathcal{T}}
\newcommand{\nuk}{\uu{n}_{K, \sigma}}
\newcommand{\Lt}{\mathcal{L}_{\mathcal{T}}}

\newcommand{\absk}{\abs{K}}
\newcommand{\abssig}{\abs{\sigma}}
\newcommand{\dsig}{D_{\sigma}}
\newcommand{\divt}{\mathrm{div}_{\T}}
\newcommand{\gradt}{\nabla_{\T}}

\newcommand{\gradd}{\nabla_{\E}}
\newcommand{\rk}{\vrho_K}

\newcommand{\sk}{{\sigma,K}}

\newcommand{\sink}{\sigma\in\E(K)}
\newcommand{\delt}{\Delta t}

\newcommand{\bfw}{\uu{w}}
\newcommand{\bfv}{\boldsymbol{v}}

\newcommand{\norm}[1]{\left\lVert#1\right\rVert}
\newcommand{\abs}[1]{\lvert#1\rvert}

\newcommand{\Div}{\mathrm{div}_x}
\newcommand{\grad}{\nabla_x}
\newcommand{\Dt}{\partial_t}

\title[Asymptotic-preserving and energy-stable method for the Euler equations]{Error Analysis of an Asymptotic-Preserving, Energy-Stable Finite Volume Method for the Barotropic Euler Equations} 
\thanks{M. A. and \ M.\ L.-M.\ were supported by the DAAD DST (German-India) Project based personnel exchange programme: Development and analysis of higher-order structure-preserving numerical methods for hyperbolic balance laws}

\author[Anandan]{M. Anandan}
\address{Institut f{\"u}r Mathematik, Johannes Gutenberg-Universit{\"a}t Mainz, Staudingerweg 9, D-55128 Mainz, Germany}
\email{manandan@uni-mainz.de}
\thanks{M. A.\ was funded by the Gutenberg Research College of the University of Mainz.}

\author[Arun]{K.R.\ Arun}
\address{School of Mathematics, Indian Institute of Science Education
  and Research Thiruvananthapuram, Thiruvananthapuram 695551, India} 
\email{arun@iisertvm.ac.in, amoghk0720@iisertvm.ac.in}
\thanks{K.R.\ A.\ sincerely thanks the Anusandhan National Research Foundation (ANRF) for the support under the Advanced Research Grant (ARG), Grant No. ANRF/ARG/2025/007027/MS}

\author[Krishnamurthy]{A.\ Krishnamurthy}

\author[Luk\'{a}\v{c}ov\'{a}-Medvi\v{d}ov\'{a}]{M. Luk\'{a}\v{c}ov\'{a}-Medvi\v{d}ov\'{a}}
\address{Institut f{\"u}r Mathematik, Johannes Gutenberg-Universit{\"a}t Mainz, Staudingerweg 9, D-55128 Mainz, Germany\\
RMU Co-Affiliate Technical University Darmstadt, Germany}
\email{lukacova@uni-mainz.de}

\thanks{M.\ L.-M.\ gratefully acknowledges the support of DFG Project 5258 3336 funded within Focused Programme SPP 2410 ``Hyperbolic Balance Laws: Complexity, Scales and Randomness" and of the Mainz Institute of Multiscale Modeling.}

\date{\today}

\subjclass{65M08, 65M12, 76M12}

\keywords{Error estimates, convergence, low Mach number limit, asymptotic preservation, energy stability, Euler equations}

\begin{document}

\begin{abstract}
We propose and analyse an energy-stable and asymptotic-preserving finite volume scheme for the compressible Euler system. Using the relative energy framework, we establish rigorous error estimates that yield convergence of the numerical solutions in two distinct regimes. For a fixed Mach number $\veps>0$, we derive error estimates between the numerical solutions and a strong solution of the compressible Euler system that are uniform with respect to the discretisation parameters, ensuring convergence as the underlying mesh is refined. In the low Mach number regime, we analyse the error between the numerical solutions and a strong solution of the incompressible Euler system and obtain asymptotic error estimates that are uniform in $\veps$ and the discretisation parameters. These results imply convergence of the numerical solutions toward a strong solution of the incompressible Euler system as $\veps$, and the discretisation parameters simultaneously tend to zero. Numerical experiments are presented to validate the theoretical analysis.
\end{abstract}

\maketitle
\par \textbf{Keywords:} Error estimates, convergence, low Mach number limit, asymptotic preservation, energy stability, Euler equations. \\ 
\par \textbf{MSC 2010 Classification:} 65M08, 65M12, 76M12.

\section{Introduction}
\label{sec:intro}

Numerical simulation of the Euler system has been extensively investigated and constitutes a fundamental tool in computational mechanics and physics; see, e.g., the monographs \cite{GR21, Tor97}. Despite the effectiveness of numerical methods in applications, rigorous convergence results remain largely restricted to scalar conservation laws; see \cite{CCL94, Kuz76, Vil94}. More refined results are available in one space dimension for specific wave
configurations, including optimal convergence rates for monotone and viscosity schemes in the presence of finitely many shocks or rarefaction waves, as well as pointwise error estimates and bounds on the thickness of numerical layers; see, e.g., \cite{TT99, TT97, TZ97}.

For multidimensional nonlinear systems of hyperbolic conservation
laws, rigorous convergence results are considerably more limited. One available result in this direction is due to \cite{JR06}, where the authors established convergence rates for numerical solutions generated by a finite volume scheme towards classical solutions, under the assumption of uniform boundedness of the numerical solutions and their Sobolev seminorms. \color{black}More recently, using the relative energy framework, \cite{FLS23,LS23,LSY22} derived convergence rates for numerical solutions towards classical solutions under weaker assumptions than those imposed in \cite{JR06}. \color{black}In the present work, \color{black}we use the relative energy framework and \color{black}derive error estimates
between numerical and strong solutions, by adopting an asymptotic-preserving (AP) finite volume scheme for the scaled Euler
system, parameterised by the Mach number $\veps$. We also derive asymptotic error estimates that are uniform with respect to the
Mach number $\veps$ as well as the discretisation parameters. It is
well known that, as $\veps \to 0$, solutions of the compressible Euler
system converge toward solutions of the incompressible Euler system;
see, e.g., \cite{KM82}. The limit $\veps \to 0$ is singular in
nature, since the purely hyperbolic compressible Euler system
converges to the mixed hyperbolic–elliptic incompressible Euler
system. Accurately capturing this regime transition at the discrete
level calls for numerical schemes that remain robust in this
singular limit. Here, the framework of AP schemes provides a natural and
effective approach.

\subsection{Background}
\subsubsection{AP schemes}
The AP framework was first introduced by Jin in the
context of kinetic models \cite{Jin99}\color{black}, and later extended to hyperbolic systems \cite{Jin2012}. \color{black}Briefly, a numerical scheme for a
parameterised system (here, the compressible Euler system) is said to be
AP if, in the limit $\veps \to 0$, it reduces to a consistent
discretisation of the limiting system (here,
 the incompressible Euler system). Moreover, the discretisation
parameters should remain uniform with respect to $\veps$. A vast body
of literature has been devoted to the development and analysis of AP
schemes; see, e.g., \cite{ADP20, AL25, ALR25, AGK23, AKL24, BAL+14, BPR13, BQRX22, DT11, DP13, HLS21, NBA+14, PR05} and the references therein.

\subsubsection{Relative energy}
The main analytical tool employed in this work is the relative energy functional, originally introduced by Dafermos \cite{Daf00}. \color{black}This framework has been widely used to compare two different solutions and plays an important role in the study of weak–strong uniqueness and singular limits of dynamical systems. We refer to the monograph by Feireisl and Novotn\'y \cite{FN17}, where the relative energy method is used to rigorously analyse various singular limits of the compressible Euler and Navier–Stokes systems. More recently, the relative energy method has been successfully applied to the convergence analysis of numerical approximations of systems such as compressible fluid equations and viscoelastic phase separation models, yielding error estimates that are uniform with respect to the discretisation parameters; see \cite{BEHL25, FLS23, GHM+15, LS23, LSY22} and the references therein. The relative energy method has also been employed in the convergence analysis of numerical solutions in the low Mach number limit of the Navier–Stokes equations \cite{FLN+18}. \color{black}


\subsection{Main contributions}
\begin{itemize}
\item We propose and analyse an asymptotic-preserving and energy-stable finite
volume scheme  for the compressible Euler
equations. To establish the stability of the scheme, we employ stabilisation
techniques following \cite{AGK23, PV16}, whereby an appropriate velocity shift is introduced in the convective fluxes of the mass and
momentum equations. 
    \item  We derive the error
estimates in two distinct regimes by applying the relative energy, a system-specific metric.
 \begin{itemize}
    \item For a fixed Mach number
$\veps > 0$, we consider the error between the numerical solutions and
a strong solution of the compressible Euler system, and derive the error
estimates that are uniform with respect to the discretisation
parameters $h, \delt$. This result yields convergence of the numerical solutions
towards the strong solution of the compressible Euler system as $h, \delt \to 0$.
\item We derive error estimates between
the numerical solutions and a strong solution of the incompressible
Euler system. These estimates are uniform in Mach number $\veps$ and the
discretisation parameters $h, \delt$. This result implies convergence of the numerical
solutions towards a strong solution of the incompressible Euler system
as $\veps,h, \delt \to 0$ simultaneously, thus proving the AP property of the method rigorously. 
\end{itemize}
\end{itemize}
The precise statements and proofs are given in \Cref{sec:err-est}, \Cref{thm:err-est}, and \Cref{thm:asymp-err-est}, after introducing the necessary notations and auxiliary results. \color{black}


\subsection{Outline}
The remainder of this paper is organised as follows. In Sections
\ref{sec:comp-eul}–\ref{sec:incomp-eul}, we introduce the compressible
and incompressible Euler systems, respectively, and recall key
local-in-time existence results for the corresponding strong
solutions. In Section \ref{sec:fv-scheme}, we present the
discretisation framework together with the finite volume scheme and
discuss its stability and consistency properties. The main results of
the paper, namely the error estimates, are derived in Section
\ref{sec:err-est}. Numerical experiments illustrating the theoretical
findings are presented in Section \ref{sec:num-res}. Finally, we
conclude with some remarks in Section \ref{sec:conc}.

\section{The Compressible Euler System and its Strong Solutions} 
\label{sec:comp-eul}

The compressible barotropic Euler system parameterised by the Mach
number \color{black}$\veps \in \mathbb{R^+}$ \color{black}on $(0,T)\times\Omega$, where $\Omega\subset\R^d$ ($d =
2,3$) is open and bounded, reads 
\begin{subequations}
\label{eqn:eul-sys}
  \begin{gather}
  \label{eqn:mss-bal}
    \partial_t \varrho_{\varepsilon} + \Div(\varrho_{\varepsilon} \bu_{\varepsilon}) = 0, \\
  \label{eqn:mom-bal}
    \partial_t (\varrho_{\varepsilon} \bu_{\varepsilon}) + \Div (\varrho_{\varepsilon} \bu_{\varepsilon} \otimes \bu_{\varepsilon} ) + \frac{1}{\varepsilon^2} \nabla_x p(\varrho_{\varepsilon}) = 0, \\
    \vrho_{\varepsilon}(0,\cdot) = \vrho_{0,\varepsilon} > 0,\quad \bu_{\varepsilon}(0,\cdot) = \bu_{0,\varepsilon}.
  \end{gather}
\end{subequations}
Here, $\varrho_{\varepsilon}$ is the fluid density, and
$\bu_{\varepsilon}$ is the fluid velocity. We consider the
barotropic equation of state for the pressure, namely, we consider
$p(\varrho)=\varrho^\gamma$ where
$\gamma > 1$ is the ratio of the specific heats, also called the
adiabatic constant. Throughout, we suppose that we work with
either space-periodic or no-flux boundary conditions.

The internal energy per unit volume or the Helmholtz functional associated with the pressure is defined as 
\begin{equation}
\label{eqn:pres-pot}
  P(z) = \frac{z^{\gamma}}{\gamma - 1},
\end{equation}
and it satisfies 
\begin{equation}
\label{eqn:pres-pot-de}
  zP^\prime(z) - p(z) = P(z).
\end{equation}
Additionally, for $z_1, z_2 > 0$, we also introduce the following relative internal energy, which is defined as 
\begin{equation}
\label{eqn:rel-pot}
  \Pi(z_1\vert z_2) = P(z_1) - P(z_2) - P^\prime(z_2)(z_1 - z_2).
\end{equation}

In order to rule out nonphysical weak solutions, the following
condition on the dissipation of the total energy is imposed as an
admissibility criterion. Namely, the inequality
\begin{equation}
\label{eqn:eng-ineq}
  \int\limits_{\Omega}\Bigl\lbrack\half\vrho_\veps\abs{\bu_\veps}^2 + \frac{1}{\veps^2}P(\vrho_\veps)\Bigr\rbrack(t,\cdot)\,\dbx\leq \int\limits_{\Omega}\Bigl\lbrack\half\vrho_\veps\abs{\bu_\veps}^2 + \frac{1}{\veps^2}P(\vrho_\veps)\Bigr\rbrack(0,\cdot)\,\dbx
\end{equation}
is appended to the weak solutions of the compressible Euler system
\eqref{eqn:eul-sys}. 

For a fixed $\veps > 0$, the following local in-time existence of
strong solutions to the compressible Euler system \eqref{eqn:eul-sys}
is known in the literature; see \cite{Kat72,Maj84}. For simplicity, we
suppress the dependence of the variables on $\veps$ in the following
theorem. 
 
\begin{theorem}[Existence of Strong Solutions]
\label{thm:strong-soln}
  Let $\Omega \subset \mathbb{R}^d$ be a bounded domain with a smooth boundary $\partial \Omega$ and let $\veps>0$ be fixed. Suppose that the initial data ($\varrho_0, \bu_0$), with $\varrho_0>0$, belong to $H^{s}(\Omega;\mathbb{R}^{d+1})$ with $s>\frac{d}{2}+1$. Then, there exists $T_{max}>0$ such that the Euler system \eqref{eqn:eul-sys} has a unique classical solution ($\varrho, \bu$) in the class
  \begin{equation}
    (\varrho,\bu) \in C
    \bigl([0,T_{max}];H^s\bigl(\Omega;\mathbb{R}^{d+1}\bigr)\bigr) \ \cap
    \ C^1
    \bigl([0,T_{max}];H^{s-1}\bigl(\Omega;\mathbb{R}^{d+1}\bigr)\bigr).    
  \end{equation} 
\end{theorem}

Assuming that the initial data have some additional regularity, we obtain the following corollary of the above proposition as a consequence of various Sobolev embedding theorems. 

\begin{corollary}
\label{cor:strong-soln}
  Suppose the assumptions of Proposition \ref{thm:strong-soln}
  hold. In addition, assume that $(\vrho_0,\bu_0)\in
  H^s(\Omega;\R^{d+1})$, where $s>\frac{d}{2} + 3$. Then, there exists
  $T_{max}>0$ such that the compressible system \eqref{eqn:eul-sys}
  has a strong solution $(\vrho,\bu)$ in the class 
  \begin{equation}
    \label{eqn:strng-reg}
    (\varrho,\bu) \in W^{2,\infty}([0,T_{max}]\times \Omega;\mathbb{R}^{d+1}).    
  \end{equation} 
\end{corollary}
\begin{proof}
  For $s>\frac{d}{2}+3>\frac{d}{2}+ 1$, we have $(\varrho,\bu)
  \in C^1
  \left([0,T_{max}];H^{s-1}\left(\Omega;\mathbb{R}^{d+1}\right)\right)$
  from Theorem \ref{thm:strong-soln}. For $s>\frac{d}{2}+2$, we have
  $(\varrho,\bu) \in W^{1,\infty} \left([0,T_{max}]\times
    \Omega;\mathbb{R}^{d+1}\right)$ since $H^{s-1}(\Omega)
  \hookrightarrow C^1(\Omega)$. Since $H^{s-1}(\Omega) \hookrightarrow
  C^2(\Omega)$ for $s>\frac{d}{2}+3$, and since the second-order time
  derivatives involve combinations of $\varrho, \bu$ and their spatial
  derivatives up to second-order, we have $\left(\varrho,\bu)
    \in W^{2,\infty} ([0,T_{max}]\times \Omega;\mathbb{R}^{d+1}\right)$.  
\end{proof}

\color{black}
\section{The Incompressible Limit of the Compressible Euler System}
\label{sec:incomp-eul}
\begin{definition}[Well-Prepared Initial Data]
\label{def:well-prep}
  We say that initial data $ (\vrho_\epso, \bu_\epso)$ are well-prepared if
  \begin{align}
    &\norm{\vrho_\epso - 1}_{L^\infty(\Omega)}\lesssim \veps^2, \label{eqn:wp-mss} \\
    &\norm{\bu_\epso-\bfv_0}_{L^2(\Omega;\R^d)} \lesssim \veps ;\quad \Div\bfv_0 = 0.  \label{eqn:wp-vel}
  \end{align}
\end{definition}
\color{black}
Assuming well-prepared initial data, cf. \eqref{eqn:wp-mss}, \eqref{eqn:wp-vel} and performing the limit
$\veps\to 0$ of the compressible Euler system \eqref{eqn:eul-sys} will
yield the incompressible Euler system; see, e.g., \cite{KM82, Maj84}
for details. The compressible density $\vrho_\veps\to r$, where $r>0$
is a constant, and the velocity field $\bu_\veps\to\bfv$ as $\veps\to
0$. Then, passing to the limit $\veps\to 0$ in the compressible Euler
equations \eqref{eqn:eul-sys}, the limit $(r,\bfv)$ will satisfy 
\begin{subequations}
\label{eqn:incomp-eul}
  \begin{gather}
    \Div\bfv = 0, \label{eqn:div-free}\\
    \Dt\bfv + \Div(\bfv\otimes\bfv) + \frac{1}{r}\grad\pi =
    0 \label{eqn:incomp-mom-bal}, \\ 
    \bfv(0,\cdot) = \bfv_0. \label{eqn:incomp-ini}
  \end{gather}
\end{subequations}
Here, $\pi$ denotes the incompressible pressure and it is the limit of $\frac{1}{\veps^2}(p(\vrho_\veps) - p(r))$ as $\veps\to 0$, and further, $\pi$ is mean-free. For the sake of simplicity, we assume that $r = 1$ throughout.  

Analogous to the compressible case, we state the following local
in-time existence result for strong solutions to the incompressible
Euler system \eqref{eqn:incomp-eul}. 

\begin{theorem}[Existence of Strong Solutions (Incompressible)]
\label{prop:incomp-str}
  Let $\Omega \subset \mathbb{R}^d$ be a bounded domain with a smooth boundary $\partial \Omega$. Suppose that the initial data $\bfv_0\in H^{s}(\Omega;\mathbb{R}^d)$ with $s>\frac{d}{2}+1$. Then, there exists $T_{max}>0$ such that the incompressible Euler system \eqref{eqn:incomp-eul} has a unique classical solution $\bfv$ in the class
  \begin{equation}
    \bfv \in C ([0,T_{max}];H^s(\Omega;\mathbb{R}^{d})) \ \cap \ C^1 ([0,T_{max}];H^{s-1}(\Omega;\mathbb{R}^{d})) \text{ and } \grad\pi \in L^{\infty} ([0,T_{max}];H^{s-1}(\Omega;\mathbb{R}^{d})).   
  \end{equation} 
\end{theorem}
\begin{proof}
  See \cite{Kat72, Maj84}. 
\end{proof}

\begin{corollary}
\label{cor:str-incomp-sol}
  Let $\Omega \subset \mathbb{R}^d$ be a bounded domain with a smooth boundary $\partial \Omega$. Suppose that the initial data $\boldsymbol{v}_0$ belong to $H^{s}(\Omega;\mathbb{R}^d)$ with $s>\frac{d}{2}+ 3$. Then, there exists $T_{max}>0$ such that the incompressible Euler system \eqref{eqn:incomp-eul} has a unique strong solution ($\pi, \boldsymbol{v}$) in the class
  \begin{equation}
  \label{SS-ICE: Regularity}
    (\pi,\boldsymbol{v}) \in W^{2,\infty} ([0,T_{max}]\times \Omega;\mathbb{R}^{d+1}).   
  \end{equation} 
\end{corollary}

\begin{proof}
  Since $s>\frac{d}{2}+3>\frac{d}{2}+1$, we have $\boldsymbol{v} \in C^1([0,T_{max}];H^{s-1}(\Omega;\mathbb{R}^{d}))$ due to Theorem \ref{prop:incomp-str}.

  For $s>\frac{d}{2}+2$, we have $\boldsymbol{v} \in C^1 ([0,T_{max}];C^1 (\Omega;\mathbb{R}^{d}))$ since $H^{s-1}(\Omega) \hookrightarrow C^1(\Omega)$. Then, we can recover $\pi$ as a solution of the elliptic problem 
  \[
  - \Delta_x \pi(t,\cdot) = \grad \boldsymbol{v}(t,\cdot) : \grad \boldsymbol{v}(t,\cdot) 
  \]
  for $t\leq T_{max}$. Due to standard elliptic estimates, we deduce $\pi \in C^1 ([0,T_{max}];W^{2,\infty}(\Omega))$.

  Next, as $s>\frac{d}{2}+3$, we have $H^{s-1}(\Omega) \hookrightarrow C^2(\Omega)$. Then, we infer that $\boldsymbol{v} \in W^{2,\infty} ([0,T_{max}]\times \Omega;\mathbb{R}^{d})$, as done earlier in Corollary \ref{cor:strong-soln}. Consequently, we can improve the regularity of $\pi$ and obtain $\pi\in W^{2,\infty} ([0,T_{max}]\times \Omega;\mathbb{R}^{d})$.
\end{proof}

\section{Finite Volume Scheme}
\label{sec:fv-scheme}

\subsection{Mesh and Unknowns}
\begin{figure}[htpb]
  \centering
  \begin{tikzpicture}
    \draw[black, thick] (0,0) rectangle (10,4);
    \draw[green!70!black, thick] (2.5, 2) -- (5,0);
    \draw[green!70!black, thick] (2.5, 2) -- (5,4);
    \fill[green!20!white] (5,0) -- (2.5, 2) -- (5,4) -- cycle;
    \draw[red, thick] (5,0) -- (7.5,2);
    \draw[red, thick] (5,4) -- (7.5,2);
    \fill[red!20!white] (5,0) -- (7.5, 2) -- (5,4) -- cycle;
    \draw[blue, thick] (5,0) -- (5,4);
    \draw[->, RedViolet, thick] (5,2) -- (6,2);
    \draw (5.5,2) node[anchor = south]{\color{RedViolet}{\footnotesize $\nuk$}};
    \filldraw [black] (2.5,2) circle (2pt) node[anchor = north east]{$x_K$};
    \filldraw [black] (5,2) circle (2pt) node[anchor = north east]{$x_\sigma$};
    \filldraw [black] (7.5,2) circle (2pt) node[anchor = north west]{$x_L$};
    \draw (1.25,2) node[anchor = south]{$K$};
    \draw (8.75,2) node[anchor = south]{$L$};
    \draw (5,0) node[anchor = north]{{\color{black}$\sigma = K\vert L$}};
    \draw (3.75,1) node[anchor = south west]{{\color{OliveGreen} $D_{K,\sigma}$}};
    \draw (6.25,1) node[anchor = south east]{{\color{red} $D_{L,\sigma}$}};
  \end{tikzpicture}
  \caption{Dual grid.}
  \label{fig:dual-grid}
\end{figure}
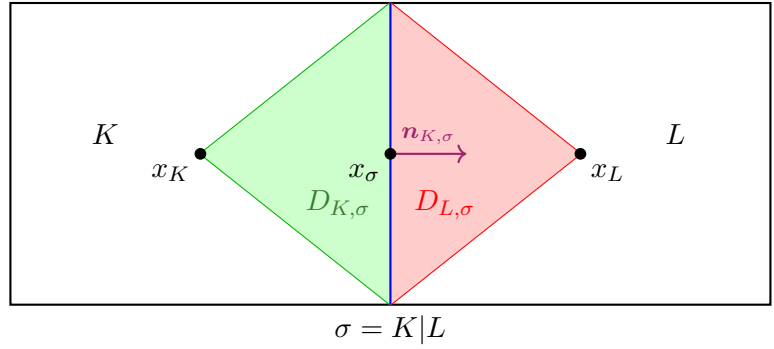
We consider a tessellation $\mathcal{T}$ of $\Omega\subset\R^d$,
consisting of closed, possibly non-uniform rectangles ($d=2$) or
closed, possibly non-uniform cuboids ($d=3$) such that
$\overline{\Omega} = \cup_{K\in\mathcal{T}} K$. The elements $K\in\T$
are referred to as primal cells or control volumes. By 
$x_K$, we denote the cell center of the primal cell
$K\in\T$. The collection of all edges ($d=2$) or faces
($d=3$) is denoted by $\mathcal{E}$, and the elements of $\E$ are denoted by $\sigma\in\E$. The sub-collections $\mathcal{E}_{int}$, $\mathcal{E}_{ext}$ and
$\mathcal{E}(K)$ denote the set of all internal edges, external edges
on the boundary $\partial \Omega$, and the edges of a cell $K \in
\mathcal{T}$, respectively. Given $K, L\in\T$, the cells are either disjoint or they share exactly one common edge, which is denoted by $\sigma=K\vert
L\in\E(K)\cap\E(L)$. For each $\sigma \in \mathcal{E}(K)$,
$\nuk$ denotes the unit normal to $\sigma$ pointing
outwards from $K$.  

The mesh size of the primal mesh $\T$ is $h = \max_{K \in \mathcal{T}} h_K$, where $h_K=\text{diam}(K)$. Throughout, we assume that there exists $h_0$ sufficiently small such that $h\in (0,h_0]$, \color{black}$h_0<1$\color{black}. We denote $a\lesssim b$ if $a\leq cb$ for a constant $c>0$ independent of the mesh parameters, and further, we write $a\approx b$ if $a\lesssim b$ and $b\lesssim a$.

Given a primal mesh $\T$ of $\Omega$, we consider the associated dual grid defined as follows. A dual cell $D_{\sigma}=D_{K,\sigma}\cup D_{L,\sigma}$ is associated to each $\sigma \in \mathcal{E}_{int}$, $\sigma = K\vert L$. Here, $D_{K,\sigma}$ (resp. $D_{L,\sigma}$) is a subset of $K$ (resp. $L$), and is shown in Figure \ref{fig:dual-grid}. Further, if $\sigma\in \E_{ext}\cap\E(K)$, we define  $D_{\sigma}=D_{K,\sigma}$. The collection of all dual cells is denoted by  $\mathcal{D}_{\mathcal{T}} = \{D_{\sigma} \}_{\sigma\in\mathcal{E}}$, and it follows that $\overline{\Omega} = \cup_{\sigma\in\E}D_\sigma$. Given $K\in\T$ and $\dsig\in\D_\T$, $\abs{K}$ and $\abs{\dsig}$, denote the respective Lebesgue measures and we suppose that $\absk\approx\abs{\dsig}\approx h^d$. Further, for $\sigma\in\E$, $\abssig$ denotes the ($d-1$)-dimensional surface measure and we assume $\abssig\approx h^{d-1}$.

By $\mathcal{L}_{\mathcal{T}}(\Omega)$, we denote the function space of scalar-valued piecewise constant functions on each cell $K \in \mathcal{T}$. The corresponding projection operator is denoted by $\Pi_{\mathcal{T}}: L^1(\Omega) \to \mathcal{L}_{\mathcal{T}}(\Omega)$, and is defined as 
\begin{equation}
  \Pi_{\mathcal{T}} q = \sum_{K \in \mathcal{T}} \bigl(\Pi_{\mathcal{T}} q\bigr)_K 1_K, \text{ with } \bigl(\Pi_{\mathcal{T}} q\bigr)_K = \frac{1}{|K|} \int_K q \dbx,
\end{equation}
Here, $1_K$ is the characteristic function of $K$. We analogously define $\mathcal{L}_{\mathcal{T}}(\Omega;\mathbb{R}^d)$, as the space of vector-valued piecewise constant functions with the projection operator being defined component-wise. We recall that the projection operator satisfies the estimate 
\begin{equation}
\label{eqn:proj-op-est}
  \norm{\Pi_\T \varphi - \varphi}_{L^p(\Omega)}\lesssim h\text{ for any }\varphi\in W^{1,\infty}(\Omega)\text{ and }1\leq p\leq\infty.
\end{equation}

For $q \in \mathcal{L}_{\mathcal{T}}(\Omega)$, $q_K$ represents the constant value of $q$ on cell $K \in \mathcal{T}$. Further, for $\sigma=K\vert L$, the average value of $q$ across $\sigma$ is defined as:
\begin{equation}
  \avg{q}_{\sigma} = \frac{q_K+q_L}{2},
\end{equation}
and this is defined component-wise for vector-valued $\boldsymbol{q} \in \mathcal{L}_{\mathcal{T}}(\Omega;\mathbb{R}^d)$. Furthermore, the jump of $q\in\Lt(\Omega)$ across $\sigma = K\vert L\in\E_{int}$ is defined as
\begin{equation}
  \diff{q}_{\sigma} = q_L - q_K,
\end{equation}
and it is considered component-wise for vector-valued $\boldsymbol{q} \in \mathcal{L}_{\mathcal{T}}(\Omega;\mathbb{R}^d)$.

We also introduce the following discrete gradient operators that are required. Given $q\in\Lt(\Omega)$, we define a discrete cell-centred gradient, denoted by $\gradt q$, which is defined as 
\begin{equation}
  \gradt q = \sum_{K\in\T}(\gradt q)_K 1_K,\quad (\gradt q)_K = \sum_{\sink}\frac{\abssig}{\absk}\avg{q}_{\sigma}\nuk.
\end{equation}
We similarly define a gradient on the dual cells, denoted by $\gradd q$ and defined as 
\begin{equation}
  \gradd q = \sum_{\sigma\in\E_{int}}(\gradd q)_{\sigma}1_{\dsig}, \quad (\gradd q)_\sigma = \frac{\abssig}{\abs{\dsig}}\diff{q}_{\sigma}\nuk.
\end{equation}
Note that for the simplicity of exposition, we have assumed that the above gradient vanishes on the dual grids corresponding to external edges.

For the ease of presentation, we suppress the dependence of the variables on $\veps$ in this section. Thus, the unknowns are the density $\vrho_K$ and the velocity $\bu_K$ for each $K\in\T$. As stated earlier, we wish to design a semi-implicit in time finite volume scheme to approximate the compressible Euler system \eqref{eqn:eul-sys}. To this end, we adopt the techniques of stabilisation, wherein an appropriate shifted velocity is introduced in the convection fluxes of the mass and momentum; see e.g.\ \cite{AGK23, AA24, PV16} and the references therein for a detailed review. Further, to derive some necessary consistency estimates, we also introduce additional viscous terms in the convective fluxes following the approach presented in \cite{FLM+21a}. With this in mind, the time interval $[0,T]$ is discretised as $0=t^0<t^1<\dots<t^N=T$, and $\Delta t = t^{n+1} - t^n$ is the constant time-step, \color{black}$\delt <1$\color{black}. The method is initialised by using the piecewise constant approximation of the initial data, namely:
\begin{equation}
\label{eqn:disc-ini}
  \varrho^0_K = \bigl(\Pi_{\mathcal{T}} \varrho_\epso\bigr)_K,\quad \bu_K^0 = \bigl(\Pi_{\mathcal{T}} \bu_\epso\bigr)_K.
\end{equation}

\subsection{Numerical Scheme}
For each $0\leq n\leq N-1$ and for each $K\in\T$, we introduce the following semi-implicit in time, fully discrete approximation of the compressible Euler system \eqref{eqn:eul-sys}:
\begin{subequations}
\label{eqn:fv-scheme}
  \begin{gather}
    \label{eqn:disc-mss}
    \frac{\varrho_{K}^{n+1}-\varrho_{K}^{n}}{\delt} + \sumintK{} F_{\sigma,K}^{n+1}  = 0,  \\
    \label{eqn:disc-mom}
    \frac{\varrho_{K}^{n+1} \bu_{K}^{n+1}-\varrho_{K}^{n}\bu_{K}^{n}}{\Delta t}+ \sumintK{}  \boldsymbol{G}_{\sigma,K}^{n+1}  + \frac{1}{\veps^2}(\gradt p^{n+1})_K = 0.
  \end{gather}
\end{subequations}
The mass and momentum convection fluxes are respectively given by
\begin{gather}
  F_{\sigma,K}^{n+1} = \vrho^{n+1}_{\sigma, up}\bfw^n_{\sigma}\cdot\nuk - \diff{\vrho^{n+1}}_{\sigma} \label{eqn:mss-flx}\\
  \boldsymbol{G}_{\sigma,K}^{n+1} = F_{\sigma,K}^{n+1} \bu_{\sigma,up}^n - \diff{\bu^n}_{\sigma}. \label{eqn:mom-flx}
\end{gather}
In the above, $\bfw_{\sigma}^n = \avg{\bu^n}_{\sigma} - \delta \bu_{\sigma}^{n+1}$ is the stabilised velocity, where $\delta\bu^{n+1}_\sigma$ is called as the stabilisation term. The choice of the stabilisation term will be made after performing a stability analysis. In addition, we denote $w_{\sigma,K}^n = \bfw_{\sigma}^n \cdot \nuk = \paraL{\avg{\bu^n}_{\sigma} - \delta \bu_{\sigma}^{n+1}} \cdot \nuk = u_{\sigma,K}^n - \delta u_{\sigma,K}^{n+1}$. The term $\vrho^{n+1}_{\sigma, up}$ is the upwind density. Namely, we write 
\begin{equation}
\label{eqn:upwd-den}
  \vrho^{n+1}_{\sigma, up}w^n_\sk = \rk^{n+1}(w^n_\sk)_{+} + \vrho^{n+1}_L(w^n_\sk)_{-}, 
\end{equation}
where the positive and negative parts of $w_{\sigma,K}^n$ are given by:
\begin{gather}
  \paraL{w_{\sigma,K}^n}_+ = \max \{ u_{\sigma,K}^n,0 \} - \min \{ \delta u_{\sigma,K}^{n+1},0 \} \geq 0, \\  
  \paraL{w_{\sigma,K}^n}_- = \min \{ u_{\sigma,K}^n,0 \} - \max \{ \delta u_{\sigma,K}^{n+1},0 \} \leq 0. 
\end{gather}
The above choice is made to ensure an upwind bias and maintain a proper signs split. In view of \eqref{eqn:upwd-den}, one can split the mass flux \eqref{eqn:mss-flx} into positive and negative parts as 
\begin{equation}
\label{eqn:mss-flx-split}
  F^{n+1}_\sk = \vrho^{n+1}_K((w^n_\sk)_+ + 1) + \vrho^{n+1}_L((w^n_\sk)_- - 1) = F^{n+1, +}_\sk + F^{n+1, -}_\sk.
\end{equation}

Finally, the velocity $\bu^n_{\sigma,up}$ appearing in the momentum flux \eqref{eqn:mom-flx} is upwind with respect to the positive and negative parts of the mass flux, namely, 
\begin{equation}
\label{eqn:upw-vel}
  F^{n+1}_{\sigma,K} \bu^n_{\sigma,up} = F^{n+1,+}_{\sigma,K} \bu^n_{K} + F^{n+1,-}_{\sigma,K} \bu^n_L.
\end{equation}

\subsection{Energy Stability} 
\label{subsec:eng-stab}

In this section, we state the energy stability property of the scheme \eqref{eqn:fv-scheme} and postpone its proof to the appendix. Analogous considerations can also be found in \cite{AGK23, AA24}.

\begin{theorem}[Energy Stability]
\label{thm:eng-stab}
Suppose the stabilisation term is chosen as 
  \begin{equation}
  \label{eqn:stab-term}
    \du^{n+1}_\sigma = \frac{\eta\delt}{\veps^2}(\gradd p^{n+1})_{\sigma}.
  \end{equation}
  Further, assume that the following conditions hold:
  \begin{enumerate}
    \item $\displaystyle\eta > \frac{3d}{2}\Bigl\{\mskip-5mu\Bigl\{ \frac{1}{\vrho^{n+1}} \Bigr\}\mskip-5mu\Bigr\}_{\sigma}$ for each $\sigma = K\vert L\in\E_{int}$;
        
    \item $\displaystyle\frac{1}{4} - \frac{\delt}{\rk^{n}}\sum_{\sink}\frac{\abssig}{\absk}\lvert F^{n+1}_\sk\rvert\geq 0$;
        
    \item $\displaystyle\frac{1}{3} - \frac{\delt}{\rk^{n+1}}\frac{\abs{\partial K}}{\absk}>\beta$, where $0<\beta<\dfrac{1}{3}$ and $\displaystyle\abs{\partial K} = \sum_{\sink}\abssig$.
  \end{enumerate}
  Then, the following local in-time energy inequality holds for each $0\leq n\leq N-1$:
  \begin{equation}
  \label{eqn:loc-eng-ineq}
    \sum_{K\in\T}\absk\Bigl(\half\rk^{n+1}\abs{\bu^{n+1}_K}^2 + \frac{1}{\veps^2}P(\rk^{n+1})\Bigr)\leq \sum_{K\in\T}\absk\Bigl(\half\rk^{n}\abs{\bu^{n}_K}^2 + \frac{1}{\veps^2}P(\rk^{n})\Bigr). 
  \end{equation}
  In addition, we also have the following entropy inequality for each $0\leq n\leq N-1$: 
  \begin{equation}
  \label{eqn:loc-ent-ineq}
    \sum_{K\in\T}\absk\Bigl(\half\rk^{n+1}\abs{\bu^{n+1}_K}^2 + \frac{1}{\veps^2}\Pi(\rk^{n+1}\vert 1)\Bigr)\leq \sum_{K\in\T}\absk\Bigl(\half\rk^{n}\abs{\bu^{n}_K}^2 + \frac{1}{\veps^2}\Pi(\rk^{n}\vert 1)\Bigr). 
  \end{equation}
\end{theorem}

Due to the choice of the stabilisation term \eqref{eqn:stab-term}, observe that the discrete mass balance \eqref{eqn:disc-mss} yields a non-linear elliptic problem for the unknown $\rk^{n+1}$. However, once $\rk^{n+1}$ is obtained, note that the velocity $\bu^{n+1}_K$ can be obtained explicitly from the discrete momentum balance \eqref{eqn:disc-mom}. In the following theorem, we establish the existence of a solution to the scheme \eqref{eqn:fv-scheme}, along with the positivity of the discrete densities. The proof follows exactly as detailed in \cite[Theorem 4.2]{AGK23}, and we refer to it for the same.

\begin{theorem}[Existence and Positivity]
  Suppose that $(\vrho^n_K, \bu^n_K)$ is known with $\vrho^n_K>0$ for each $K\in\T$. Then, there exists a solution $(\vrho^{n+1}_K, \bu^{n+1}_K)$ to the scheme \eqref{eqn:fv-scheme}, such that $\vrho^{n+1}_K>0$ for each $K\in\T$.
\end{theorem}

\subsection{Consistency}
\label{subsec:cons}

For the ease of exposition, we introduce the following notations. We set
\begin{equation}
  \vrho_{h,\delt,\veps} = \sum_{n = 0}^{N-1}\sum_{K\in\T}\vrho^n_K 1_K 1_{[t^n, t^{n+1})},\quad \bu_{h,\delt,\veps} = \sum_{n = 0}^{N-1}\sum_{K\in\T}\bu^n_K 1_K 1_{[t^n, t^{n+1})}.
\end{equation}

To prove the consistency property of the present scheme, we need to assume uniform bounds on the discrete functions $\vrho_{h,\delt,\veps}$ and $\bu_{h,\delt,\veps}$.  

\begin{hypothesis}
\label{hyp:bnd}
  There exist constants $0<\underline{\vrho}\leq\overline{\vrho}$ and $\overline{u} > 0$ that are independent of the mesh parameters and $\veps$ such that 
  \begin{equation}
  \label{eqn:bnd-hyp}
    0<\underline{\vrho}\leq\vrho_{h,\delt,\veps}\leq\overline{\vrho},\quad \lvert\bu_{h,\delt,\veps}\rvert\leq\overline{u}.
  \end{equation}
\end{hypothesis}

We now state the consistency of the scheme \eqref{eqn:fv-scheme} and postpone the proof to the appendix.

\begin{theorem}[Consistency]
\label{thm:cons}
  Assume that the conditions required for Theorem \ref{thm:eng-stab} and Hypothesis \ref{hyp:bnd} hold true. Then, the numerical solutions generated by the scheme are consistent with the Euler system \eqref{eqn:eul-sys}, i.e.\
  \begin{equation}
  \label{eqn:cons-mass}
    \biggl\lbrack\int\limits_\Omega\vrho_{h,\delt,\veps}(t,\cdot)\varphi(t,\cdot)\,\dbx\biggr\rbrack_{t=0}^{t=\tau} = \lint\lbrack \vrho_{h,\delt,\veps}\Dt\varphi + \vrho_{h,\delt,\veps}\bu_{h,\delt,\veps}\cdot\grad\varphi\rbrack\dbx\,\dt + \mathcal{S}^{mass}_{h,\delt},
  \end{equation}
  for any $\varphi\in W^{2,\infty}(\odom)$ and $\tau\in(0,T)$; 
  \begin{equation}
  \label{eqn:cons-mom}
    \begin{split}
      \biggl\lbrack\int\limits_\Omega\vrho_{h,\delt,\veps}\bu_{h,\delt,\veps}(t,\cdot)\cdot\uphi(t,\cdot)\,\dbx\biggr\rbrack_{t=0}^{t=\tau} = \lint\Bigl\lbrack\vrho_{h,\delt,\veps}\bu_{h,\delt,\veps}\cdot\Dt\uphi + \vrho_{h,\delt,\veps}(\bu_{h,\delt,\veps}\otimes \bu_{h,\delt,\veps})\colon\grad\uphi \\
      + \frac{1}{\veps^2}p_{h,\delt,\veps}\,\Div\uphi\Bigr\rbrack\,\dbx\,\dt + \mathcal{S}^{mom}_{h,\delt},
    \end{split}
  \end{equation}
  for any $\uphi\in W^{2,\infty}(\odom;\R^d)$ and $\tau\in (0,T)$.
  
  The consistency errors are such that 
  \begin{equation}
  \label{eqn:cons-err}
    \begin{split}
    \lvert \mathcal{S}^{mass}_{h,\delt}\rvert\lesssim C_\varphi (\sqrt{\delt} + \sqrt{\color{black}h\color{black}}), \       \lvert\mathcal{S}^{mom}_{h,\delt}\rvert\lesssim C_{\uphi} (\sqrt{\delt} + \sqrt{\color{black}h\color{black}}) , \text{ for } \veps=1 \\
      \lvert \mathcal{S}^{mass}_{h,\delt}\rvert\lesssim C_\varphi (1+\veps)(\sqrt{\delt} + \sqrt{\color{black}h\color{black}}), \       \lvert\mathcal{S}^{mom}_{h,\delt}\rvert\lesssim C_{\uphi} (1+\veps)(\sqrt{\delt} + \sqrt{\color{black}h\color{black}}), \text{ for } \veps<1
    \end{split}
  \end{equation}
  for constants $C_\varphi, C_{\uphi}>0$ that are only dependent on the test functions.
\end{theorem}

\section{Error Estimates}
\label{sec:err-est}

The goal of this section is to derive uniform error estimates that will yield the convergence rates of the numerical solutions $(\vrho_{h,\delt,\veps}, \bu_{h,\delt,\veps})$ generated by the scheme \eqref{eqn:fv-scheme} for the following two cases. In the first case, we fix $\veps$ and estimate the difference between the numerical solutions and a strong solution $(\tvrho, \tu)$ of the compressible Euler system \eqref{eqn:eul-sys}. The estimates obtained in this case are uniform with respect to the mesh parameters $h, \delt$, and will yield the convergence of numerical solutions towards the strong solution of the compressible Euler system as the mesh parameters vanish. In the second case, we consider a strong solution $\tv$ of the incompressible Euler system, and we proceed to derive the asymptotic error estimates that are uniform in $\veps, h$ and $\delt$. These estimates will, in turn, guarantee the convergence of the discrete solutions towards the strong solution of the incompressible system as $\veps$, and the mesh parameters simultaneously go to zero.

\subsection{Error Estimates in the Compressible Regime}
\label{subsec:comp-err-est}

In this section, we detail the error estimates obtained between the numerical solutions $(\vrho_{h,\delt,\veps}, \bu_{h,\delt,\veps})$ and a strong solution $(\tvrho, \tu)$ of the compressible Euler system \eqref{eqn:eul-sys} for a \color{black}finite $\veps>0$. Consequently, without loss of generality, we consider $\veps = 1$ \color{black}and denote $\vrho_{h,\delt, 1} = \vrho_{h,\delt}$, and analogously define $\bu_{h,\delt}$. To begin with, we introduce the relative energy between the numerical solution and the strong solution via
\begin{equation}
\label{eqn:rel-eng}
  E_{rel}(\vrho_{h,\delt}, \bu_{h,\delt}\vert \tvrho, \tu) = \half\vrho_{h,\delt}\lvert\bu_{h,\delt} - \tu\rvert^2 + P(\vrho_{h,\delt}) - P(\tvrho) - P^\prime(\tvrho)(\vrho_{h,\delt} - \tvrho).
\end{equation}

Observe that since $P$ is a convex function, cf.\ \eqref{eqn:pres-pot}, the relative energy is non-negative. Next, we recall that the existence of a strong solution to the Euler system is given by Theorem \ref{thm:strong-soln} and Corollary \ref{cor:strong-soln}, provided the initial datum is sufficiently regular. In particular, if $(\tvrho_0, \tu_0)\in H^s(\Omega;\R^{d+1})$ where $s>d/2 + 3$, then the strong solution will have the regularity $(\tvrho, \tu)\in W^{2,\infty}(\odom; \R^{d+1})$; cf.\ Corollary \ref{cor:strong-soln}.   

In view of Hypothesis \ref{hyp:bnd}, we have the following straightforward equivalence.
\begin{lemma}
\label{lem:rel-eng-equiv}
  Let $(\widetilde{\varrho},\widetilde{\bu})$ be the strong solution of the Euler system \eqref{eqn:eul-sys}, and let $(\varrho_{h,\Delta t},\bu_{h,\Delta t})$ be the numerical solution obtained by the scheme \eqref{eqn:fv-scheme}. Assuming Hypothesis \ref{hyp:bnd}, we have the following equivalence
  \begin{equation}
    \label{eqn:comp-equiv}
    E_{rel}(\varrho_{h,\Delta t},\bu_{h,\Delta t}|\widetilde{\varrho},\widetilde{\bu}) \approx \modL{\varrho_{h,\Delta t} - \widetilde{\varrho}}^2 + \modL{\varrho_{h,\Delta t} \bu_{h,\Delta t} - \widetilde{\varrho}\widetilde{\bu}}^2.
  \end{equation}
\end{lemma}
Consequently, for $\tau\in(0,T)$, integrating \eqref{eqn:comp-equiv} over $\Omega$ yields
\begin{equation}
\label{eqn:comp-equiv-2}
  \begin{split}
    \int\limits_{\Omega}E_{rel}(\varrho_{h,\Delta t},\bu_{h,\Delta t}|\widetilde{\varrho},\widetilde{\bu})(\tau,\cdot)\,\dbx \approx &\norm{\vrho_{h,\delt}(\tau,\cdot) - \tvrho(\tau,\cdot)}_{L^2(\Omega)}^2 \\
    &+ \norm{\vrho_{h,\delt}\bu_{h,\delt}(\tau,\cdot) - \tvrho\tu(\tau,\cdot)}_{L^2(\Omega;\R^d)}^2 
  \end{split}
\end{equation}
\color{black}\begin{proof}
    See the proof of Lemma 2.7 in \cite{LSY22}. 
\end{proof}\color{black}
Thus, having control on the relative energy on the left-hand side of the above equation will give us control on the norms of the differences between the numerical and strong solutions. This will further yield the convergence of the numerical solutions towards the strong solutions and provide us with the rate of convergence in the $L^2$-norm. To this end, we aim to establish a discrete variant of the relative energy inequality, which will in turn allow us to obtain a control on the total relative energy; see e.g.\ \cite{BLM+23, FLS23, LSY22} and the references therein.

We are now ready to state and prove the first main result of this paper. 

\begin{theorem}[Error Estimates]
\label{thm:err-est}
  Let the initial data $(\tvrho_0,\tu_0)\in H^s(\Omega;\R^{d+1})$, where $s> d/2 +3$. Let $(\tvrho,\tu)\in W^{2,\infty}(\odom;\R^{d+1})$ be the strong solution emanating from $(\tvrho_0,\tu_0)$, where $T\leq T_{max}$. Let the finite volume scheme \eqref{eqn:fv-scheme} be initialised as 
  \begin{equation}
  \label{eqn:ini-proj}
    \vrho^0_K = (\Pi_\T \tvrho_0)_K,\quad \bu^0_K = (\Pi_\T \tu_0)_K,\ K\in\T,
  \end{equation}
  and let $(\vrho_{h,\delt}, \bu_{h,\delt})$ denote the corresponding numerical  solution. Finally, let the assumptions of Theorem \ref{thm:cons} hold true. Then, for any $\tau\in (0,T)$, there exists a constant 
  \[
    C = C(\norm{\tvrho}_{W^{2,\infty}}, \norm{\tu}_{W^{2,\infty}}) > 0 
  \]
  such that
  \begin{equation}
  \label{eqn:err-est}
    \int\limits_\Omega E_{rel}(\vrho_{h,\delt},\bu_{h,\delt}\,\vert\, \tvrho,\tu)(\tau,\cdot)\,\dbx\lesssim C(\sqrt{h _\T} + \sqrt{\delt}).
  \end{equation}
\end{theorem}

\begin{proof}
  The proof can be split into two steps:
  \begin{itemize}
    \item Firstly, we take suitable combinations of the strong solution $\paraL{\widetilde{\varrho},\widetilde{\bu}}$ as test functions in the consistency formulation \eqref{eqn:cons-mass}-\eqref{eqn:cons-mom}, and derive the relative energy inequality between the numerical solution $\paraL{\varrho_{h,\Delta t},\bu_{h,\Delta t}}$ and the strong $\paraL{\widetilde{\varrho},\widetilde{\bu}}$. 
    
    \item Secondly, we estimate the right-hand side of the resulting relative energy inequality such that all the terms on the right-hand side are bounded by the discretisation parameters $h, \Delta t$, or by the relative energy itself. Finally, we invoke the \textit{Gronwall's lemma} to obtain the estimate \eqref{eqn:err-est}.
    \end{itemize}

  \textit{Step 1:} Since the initial data $(\tvrho_0,\tu_0)\in H^s(\Omega;\R^{d+1})$, where $s> d/2 + 3$, Corollary \ref{cor:strong-soln} guarantees the existence of $(\tvrho,\tu)\in W^{2,\infty}(\odom;\R^{d+1})$ such that 
  \begin{subequations}
  \label{eqn:strong-eul-sol}
    \begin{gather}
      \Dt \tvrho + \Div(\tvrho\tu) = 0, \\
      \Dt(\tvrho\tu) + \Div(\tvrho\tu\otimes\tu) + \grad p(\tvrho) = 0.
    \end{gather}
  \end{subequations}
  In particular, we see that
  \begin{equation}
  \label{eqn:strong-vel}
    \Dt\tu + (\tu\cdot\grad)\tu + \frac{1}{\tvrho}\grad p(\tvrho) = 0.
  \end{equation}
  Next, since the scheme is initialised using the initial data $(\tvrho_0, \tu_0)$, cf.\ \eqref{eqn:ini-proj}, observe that as a consequence of \eqref{eqn:proj-op-est} and \eqref{eqn:comp-equiv-2}, we get
  \begin{equation}
  \label{eqn:err-est-ini}
    \int\limits_{\Omega}E_{rel}(\vrho_{h,\delt},\bu_{h,\delt}\,\vert\,\tvrho, \tu)(0,\cdot)\,\dbx \lesssim h^2.
  \end{equation}

  We now proceed to establish the relative energy inequality. Firstly, we rewrite the relative energy \eqref{eqn:rel-eng} in a more convenient form as 
  \begin{equation}
  \label{eqn:err-est-1}
    \begin{split}
      E_{rel}(\vrho_{h,\delt},\bu_{h,\delt}\,\vert\,\tvrho, \tu) = &\Bigl(\half\vrho_{h,\delt}\lvert\bu_{h,\delt}\rvert^2 + P(\vrho_{h,\delt})\Bigr) + \Bigl(\frac{1}{2} \lvert{\tu}\rvert^2  -  P'(\tvrho)\Bigr)\varrho_{h,\Delta t} \\
      &- \varrho_{h,\Delta t}\bu_{h,\Delta t}\cdot \tu + p(\widetilde{\varrho}).
    \end{split}
  \end{equation}
  Next, for $\tau\in(0,T)$, we calculate the time increments $[\int_{\Omega} E(\varrho_{h,\Delta t},\bu_{h,\Delta t}|\widetilde{\varrho},\widetilde{\bu}) \dbx]_{t=0}^{t=\tau}$. From \eqref{eqn:err-est-1}, we get 
 
  \begin{equation}
  \label{eqn:err-est-2}
    \begin{split}
      \left[ \int_{\Omega} E_{rel}(\varrho_{h,\delt},\bu_{h,\delt}|\widetilde{\varrho},\widetilde{\bu}) \dbx \right]_{t=0}^{t=\tau} = &\left[ \int_{\Omega}\biggl(\half\vrho_{h,\delt}\abs{\bu_{h,\delt}}^2 + P(\vrho_{h,\delt})\biggr)\dbx\right]_{t=0}^{t = \tau} \\
      &+ \left[ \int_{\Omega} \varrho_{h,\delt} \paraL{{\frac{1}{2} \modL{\widetilde{\bu}}^2 - P'(\widetilde{\varrho})}} \dbx \right]_{t=0}^{t=\tau} \\
      &- \left[ \int_{\Omega}  \varrho_{h,\delt}\bu_{h,\delt} \cdot \widetilde{\bu} \dbx\right]_{t=0}^{t=\tau}  
      + \left[ \int_{\Omega} p\paraL{\widetilde{\varrho}} \dbx \right]_{t=0}^{t=\tau}.
    \end{split}
  \end{equation}
  Invoking the stability of the scheme, cf.\ Theorem \ref{thm:eng-stab}, observe that the first term on the right-hand side of the above equation is non-positive. Next, owing to the regularity of the strong solution $(\tvrho, \tu)$, we can take $\varphi = {\frac{1}{2} \modL{\widetilde{\bu}}^2 - P'(\widetilde{\varrho})}$ and $\boldsymbol{\varphi} = \widetilde{\bu}$ as test functions in the consistency formulation \eqref{eqn:cons-mass} of the mass and \eqref{eqn:cons-mom} of the momentum. Straightforward simplification will yield
  \begin{multline}
  \label{eqn:err-est-3}
    \left[ \int_{\Omega} E_{rel}(\varrho_{h, \delt},\bu_{h, \delt}|\widetilde{\varrho},\widetilde{\bu}) \dbx \right]_{t=0}^{t=\tau} \leq  \left[ \int_{\Omega} p\paraL{\widetilde{\varrho}}\,\dbx\right]_{t=0}^{t=\tau} + \mathcal{S}_{h,\Delta t}^{mass} - \mathcal{S}_{h,\Delta t}^{mom} \\ 
    + \int_0^{\tau} \int_{\Omega} \bigl[\varrho_{h, \delt} \paraL{\widetilde{\bu}-\bu_{h, \delt}} \cdot \partial_t \widetilde{\bu}  + \varrho_{h, \delt} \paraL{\widetilde{\bu}-\bu_{h, \delt}} \otimes \bu_{h, \delt} : \grad \widetilde{\bu}\bigr]\,\dbx\,\dt \\
    - \int_0^{\tau} \int_{\Omega} \bigl[ P''(\widetilde{\varrho}) \paraL{\varrho_{h, \delt}\partial_t \widetilde{\varrho} + \varrho_{h, \delt} \bu_{h, \delt} \cdot \grad \widetilde{\varrho}}  - p_{h, \delt} \Div \widetilde{\bu} \bigr]  \dbx\,\dt,
  \end{multline}
  with 
  \begin{gather}
    \modL{\mathcal{S}_{h,\Delta t}^{mass}} \leq  C_1\paraL{ \norm{\frac{1}{2} \modL{\widetilde{\bu}}^2 - P'(\widetilde{\varrho})}_{W^{2,\infty}\paraL{[0,\tau]\times \Omega}} }  \paraL{\sqrt{h}+\sqrt{\Delta t} }, \\ 
    \modL{\mathcal{S}_{h,\Delta t}^{mom}} \leq  \ C_2 \paraL{\norm{\widetilde{\bu}}_{W^{2,\infty}\paraL{[0,\tau]\times\Omega; \mathbb{R}^d}}}  \paraL{\sqrt{h}+\sqrt{\Delta t}},
  \end{gather}
  cf. \eqref{eqn:cons-err}.
  We proceed to handle all the terms appearing in the right-hand side of \eqref{eqn:err-est-3} separately. Firstly, note that we can write the integrand appearing in the fourth term of \eqref{eqn:err-est-3} as
  \begin{equation}
  \label{eqn:err-est-4}
    \begin{split}
      \vrho_{h,\delt}(\tu - \bu_{h,\delt})\cdot\Dt\tu + \vrho_{h,\delt}(\tu - \bu_{h,\delt})\otimes \bu_{h,\delt}\colon\grad\tu \\
      = -\frac{p^\prime(\tvrho)}{\tvrho}\vrho_{h,\delt}(\tu - \bu_{h,\delt})\cdot\grad\tvrho - \vrho_{h,\delt}(\bu_{h,\delt} - \tu)\otimes(\bu_{h,\delt} - \tu)\colon\grad\tu,
    \end{split}
  \end{equation}
  wherein we have used the fact that $(\tvrho, \tu)$ satisfy \eqref{eqn:strong-eul-sol}. Next, we write the first term in \eqref{eqn:err-est-3} as
  \begin{equation}
  \label{eqn:err-est-5}
    \Bigl\lbrack\int_{\Omega}p(\tvrho)\,\dbx\Bigr\rbrack_{t=0}^{t=\tau} = \int_0^\tau\int_{\Omega} p^\prime(\tvrho)\Dt\tvrho\,\dbx\,\dt.
  \end{equation}
  Now, noting that $\int_0^\tau\int_{\Omega}\Div(p(\tvrho)\tu)\, \dbx\,\dt = 0$, we can freely add it to \eqref{eqn:err-est-3}. Finally, we write $P^{\prime\prime}(\tvrho) = p^\prime(\tvrho)/\tvrho$ in the last term of \eqref{eqn:err-est-3}. Thus, summarising the above discussion, from \eqref{eqn:err-est-3}-\eqref{eqn:err-est-5}, we obtain 
\begin{multline}
    \left[ \int_{\Omega} E_{rel}(\varrho_{h, \delt},\bu_{h, \delt}\vert \widetilde{\varrho},\widetilde{\bu}) \dbx \right]_{t=0}^{t=\tau}  \leq - \int_{0}^{\tau} \int_{\Omega} \varrho_{h, \delt} \paraL{\bu_{h, \delt} - \widetilde{\bu}} \otimes \paraL{\bu_{h, \delt} - \widetilde{\bu}} : \grad \widetilde{\bu}\, \dbx\,\dt \\ 
    + \int_{0}^{\tau} \int_{\Omega} p^\prime(\tvrho)\biggl(1 - \frac{\vrho_{h,\delt}}{\tvrho}\biggr)\lbrack\Dt\tvrho + \tu\cdot\grad\tvrho\rbrack\,\dbx\,\dt - \int_0^\tau\int_\Omega(p_{h,\delt} - p(\tvrho))\Div\tu\,\dbx\,\dt + \mathcal{S}_{h,\Delta t}^{mass} - \mathcal{S}_{h,\Delta t}^{mom}.  
  \end{multline}

  Finally, writing $\Dt\tvrho + \tu\cdot\grad\tvrho = - \tvrho\,\Div\tu$ in the above equation, we arrive at the desired form of the relative energy inequality.
  \begin{multline}
  \label{eqn:err-est-6}
    \left[ \int_{\Omega} E_{rel}(\varrho_{h, \delt},\bu_{h, \delt}\vert \widetilde{\varrho},\widetilde{\bu}) \dbx \right]_{t=0}^{t=\tau}  \leq - \int_{0}^{\tau} \int_{\Omega} \varrho_{h, \delt} \paraL{\bu_{h, \delt} - \widetilde{\bu}} \otimes \paraL{\bu_{h, \delt} - \widetilde{\bu}} : \grad \widetilde{\bu}\, \dbx\,\dt \\ 
    - \int_{0}^{\tau} \int_{\Omega} \paraL{p(\varrho_{h, \delt}) - p'(\widetilde{\varrho})(\varrho_{h, \delt}-\widetilde{\varrho}) - p(\widetilde{\varrho})} \Div \widetilde{\bu}\,\dbx\,\dt + \mathcal{S}_{h,\Delta t}^{mass} - \mathcal{S}_{h,\Delta t}^{mom}.  
  \end{multline}
  
  \textit{Step 2:} We proceed to estimate the right-hand side of the above equation and write it in a suitable form to apply Gronwall's inequality. Noting that $\varrho_{h,\delt}$ is bounded from above due to Hypothesis \ref{hyp:bnd}, the first term of \eqref{eqn:err-est-6} can be estimated after invoking the equivalence \eqref{eqn:comp-equiv-2} as,
  \begin{gather*}
    \modL{\int_{0}^{\tau} \int_{\Omega} \varrho_{h,\delt} \paraL{\bu_{h, \delt} - \widetilde{\bu}} \otimes \paraL{\bu_{h, \delt} - \widetilde{\bu}} \colon \grad \widetilde{\bu} \dbx\dt} \\
      \lesssim \norm{\grad \widetilde{\bu}}_{L^{\infty}\paraL{[0,\tau]\times \Omega;\mathbb{R}^{d\times d}}} \int_{0}^{\tau} \int_{\Omega} E_{rel}(\varrho_{h, \delt},\bu_{h, \delt}|\widetilde{\varrho},\widetilde{\bu}) \dbx \dt. 
  \end{gather*}
  Since the second term on the right-hand side of \eqref{eqn:err-est-6} constitutes a first-order Taylor expansion, we obtain,
  \begin{multline*}
    \modL{\int_{0}^{\tau} \int_{\Omega} \paraL{p(\varrho_{h, \delt}) - p'(\widetilde{\varrho})(\varrho_{h, \delt}-\widetilde{\varrho}) - p(\widetilde{\varrho})} \Div \widetilde{\bu} \dbx\dt} \\
    \lesssim \norm{\grad \widetilde{\bu}}_{L^{\infty}\paraL{[0,\tau]\times \Omega;\mathbb{R}^{d\times d}}} \int_{0}^{\tau} \norm{\varrho_{h, \delt}-\widetilde{\varrho}}^2_{L^2\paraL{\Omega}} \dt \nonumber \\
    \lesssim \norm{\grad \widetilde{\bu}}_{L^{\infty}\paraL{[0,\tau]\times \Omega;\mathbb{R}^{d\times d}}} \int_{0}^{\tau} \int_{\Omega} E_{rel}(\varrho_{h, \delt},\bu_{h, \delt}|\widetilde{\varrho},\widetilde{\bu}) \dbx\dt. 
  \end{multline*}
  Thus, by using the bounds on consistency errors, the relative energy inequality \eqref{eqn:err-est-6} becomes,
  \begin{multline}
    \left[ \int_{\Omega} E_{rel}(\varrho_{h, \delt},\bu_{h, \delt}|\widetilde{\varrho},\widetilde{\bu}) \dbx \right]_{t=0}^{t=\tau}  
    \leq c_1\Bigg(\norm{\frac{1}{2} \modL{\widetilde{\bu}}^2 - P'(\widetilde{\varrho})}_{W^{2,\infty}\paraL{[0,\tau]\times \Omega}}, \norm{\widetilde{\bu}}_{W^{2,\infty}\paraL{[0,\tau]\times\Omega; \mathbb{R}^d}}\Bigg) \\
    \paraL{\sqrt{h}+\sqrt{\Delta t}} 
    + c_2 \paraL{\norm{\grad \widetilde{\bu}}_{L^{\infty}\paraL{[0,\tau]\times \Omega;\mathbb{R}^{d\times d}}} } \int_{0}^{\tau} \int_{\Omega} E_{rel}(\varrho_{h, \delt},\bu_{h, \delt}|\widetilde{\varrho},\widetilde{\bu}) \dbx\dt.
  \end{multline}
  Finally, invoking Gronwall's lemma and recalling the estimate on the initial relative energy \eqref{eqn:err-est-ini} will immediately yield the required estimate \eqref{eqn:err-est}.
\end{proof}

A straightforward corollary of the above result now follows as a consequence of the equivalence \eqref{eqn:comp-equiv-2}, which yields the convergence rates of the numerical solutions towards the strong solution.
\begin{corollary}
  Let the conditions given in Theorem \ref{thm:err-est} hold.  Then, for any $\tau\in(0,T)$
  \begin{align}
    \norm{\vrho_{h,\delt}(\tau,\cdot) - \tvrho(\tau,\cdot)}_{L^2(\Omega)}\lesssim h^{\frac{1}{4}} + \delt^{\frac{1}{4}}, \\
    \norm{\vrho_{h,\delt}\bu_{h,\delt}(\tau,\cdot) - \tvrho\tu(\tau,\cdot)}_{L^2(\Omega;\R^d)}\lesssim h^{\frac{1}{4}} + \delt^{\frac{1}{4}}.
  \end{align}
\end{corollary}

\subsection{Error Estimates in the Asymptotic Regime}
\label{subsec:asymp-err-est}

In this section, we present the error estimates \color{black}for a singular limit problem, i.e. we compare the error \color{black}between the numerical solutions $(\vrho_{h,\delt,\veps}, \bu_{h,\delt,\veps})$ and a strong solution $\tv$ of the incompressible Euler system \eqref{eqn:incomp-eul}. To begin, we reintroduce the relative energy between the numerical solution and the strong solution, which is defined as 
\begin{equation}
\label{eqn:rel-eng-eps}
  E_{rel}(\vrho_{h,\delt,\veps}, \bu_{h,\delt,\veps}\vert 1, \tv) = \half\vrho_{h,\delt,\veps}\lvert\bu_{h,\delt,\veps} - \tv\rvert^2 + \frac{1}{\veps^2}\Pi(\vrho_{h,\delt,\veps}\vert 1),
\end{equation}
where $\Pi$ is defined in \eqref{eqn:rel-pot}. The existence of the strong solution $\tv\in W^{2,\infty}(\Omega;\R^d)$ to the incompressible Euler system is guaranteed provided the initial datum $\tv_0\in H^s(\Omega;\R^d)$ with $\Div\tv_0 = 0$, where $s>d/2+3$, cf.\ Corollary \ref{cor:str-incomp-sol}. Now, we recall that the initial data $(\vrho_\epso, \bu_\epso)$ used to initialise the scheme \eqref{eqn:fv-scheme} is well-prepared, i.e.\ 
\begin{align}
  &\norm{\vrho_\epso - 1}_{L^\infty(\Omega)}\lesssim \veps^2, \label{eqn:wp-mss-1} \\
  &\norm{\bu_\epso-\tv_0}_{L^2(\Omega;\R^d)} \lesssim \veps ;\quad \Div\tv_0 = 0. \label{eqn:wp-vel-1}
\end{align}
Under the above assumption, we have the following estimates on $\vrho_{h,\delt,\veps} - 1$, which is required for the forthcoming analysis. 

\color{black}
\begin{lemma}
\label{lem:disc-den-conv}
    Suppose that the initial data $(\vrho_\epso,\bu_\epso)$ are well-prepared, cf. \eqref{eqn:wp-mss-1}, \eqref{eqn:wp-vel-1}. Then 
    \begin{equation}
    \label{WP data_num}
        \norm{\frac{ \varrho_{h,\Delta t,\veps} -1}{\varepsilon}}_{L^{\infty}(0,T;L^2(\Omega))} \lesssim 1.
    \end{equation}
    In particular, the RHS of \eqref{WP data_num} is independent of $\varepsilon$ and the mesh and time-step parameters $h, \delt$, respectively. 
\end{lemma}
\begin{proof}
    From the renormalized total energy inequality \eqref{eqn:loc-ent-ineq}, we obtain for $\tau \in [0,T]$, 
    \begin{multline*}
        \frac{1}{\varepsilon^2}   \int_{\Omega}  \paraL{ P(\varrho_{h,\Delta t,\veps}) - P'(1)\paraL{\varrho_{h,\Delta t,\veps} - 1} - P(1)} \dx{} \ (\tau) \\ \leq \frac{1}{\varepsilon^2}   \int_{\Omega}  \paraL{ P(\varrho_{h,\Delta t,\veps}) - P'(1)\paraL{\varrho_{h,\Delta t,\veps} - 1} - P(1)} \dx{} \ (0) \\
        \lesssim \int_\Omega P''(\varrho^*) \varepsilon^2 \dx{}  \text{ (due to well-prepared initial data \eqref{eqn:wp-mss})} \lesssim 1. 
    \end{multline*}
    Thus,
    \begin{equation*}
        \frac{1}{2\varepsilon^2}   \int_{\Omega}  P''(\varrho_{\dagger}) \paraL{\varrho_{h,\Delta t,\veps} - 1}^2 \dx{} \ (\tau) \leq C, \text{uniformly as } \varepsilon \to 0,  \text{ with } \varrho_{\dagger} \in [\varrho_{h,\Delta t,\veps},1].
    \end{equation*}
    Consequently,
    \begin{equation}
        \frac{1}{\varepsilon^2} \int_{\Omega} \paraL{\varrho_{h,\Delta t,\veps} - 1}^2 \dx{} \ (\tau) \lesssim 1, \text{uniformly as } \varepsilon \to 0.
    \end{equation}
    Note that $P$ is a convex function and $P''(\varrho_{\dagger})= \kappa \gamma \varrho_{\dagger}^{\gamma -2}$ with $\varrho_{\dagger} \in [\varrho_{h,\Delta t,\veps},1]$ is a positive quantity bounded from below, cf. Hypothesis \ref{hyp:bnd}. This finishes the proof.
\end{proof}
\color{black}

Analogously to the compressible case, we have the following equivalence as a straightforward consequence of Hypothesis \ref{hyp:bnd}.
\begin{proposition}
\label{prop:rel-eng-equiv-eps}
  Let $\tv\in W^{2,\infty}(\odom;\R^d)$ be a strong solution of the incompressible Euler system \eqref{eqn:incomp-eul}. Let $(\vrho_{h,\delt,\veps},\bu_{h,\delt,\veps})$ be the numerical solution generated by the scheme \eqref{eqn:fv-scheme}. Then, the following holds true:
  \begin{equation}
  \label{eqn:rel-eng-equiv-incomp}
    E_{rel}(\vrho_{h,\delt,\veps},\bu_{h,\delt,\veps}\vert 1,\tv) \approx \frac{1}{\veps^2}\abs{\vrho_{h,\delt,\veps} - 1}^2 + \abs{\vrho_{h,\delt,\veps}\bu_{h,\delt,\veps} - \tv}^2.
  \end{equation}
\end{proposition}
In particular, we obtain for $\tau\in(0,T)$,
\begin{equation}
\label{eqn:rel-eng-equiv-incomp-2}
  \begin{split}
    \int\limits_{\Omega}E_{rel}(\vrho_{h,\delt,\veps},\bu_{h,\delt,\veps}\vert 1,\tv)(\tau,\cdot)\,\dbx\approx&\frac{1}{\veps^2}\norm{\vrho_{h,\delt,\veps}(\tau,\cdot) - 1}^2_{L^2(\Omega)} \\
    &+ \norm{\vrho_{h,\delt,\veps}\bu_{h,\delt,\veps}(\tau,\cdot) - \tv(\tau,\cdot)}^2_{L^2(\Omega;\R^d)}.
  \end{split}
\end{equation}
As done in the previous section, we set up a discrete variant of the relative energy inequality while using the strong solution $\tv$ as a suitable test function in the consistency formulation \eqref{eqn:cons-mass}-\eqref{eqn:cons-mom}. We now proceed to state and prove the second main result of this paper.

\begin{theorem}[Asymptotic Error Estimates]
\label{thm:asymp-err-est}
  Let the initial data $(\vrho_\epso, \bu_\epso)$ for the compressible Euler system \eqref{eqn:eul-sys} be well-prepared. Further, assume that $\tv_0\in H^s(\Omega;\R^d)$ with $\Div\tv_0 = 0$, where $s> d/2 + 3$. Let $\tv\in W^{2,\infty}(\odom;\R^d)$ be the strong solution of the incompressible system  emanating from $\tv_0$, where $T<T_{max}$. Let $(\vrho_{h,\delt,\veps}, \bu_{h,\delt,\veps})$ be the numerical solution generated by the scheme \eqref{eqn:fv-scheme}. Then, for any $\tau\in(0,T)$, there exists a constant 
  \[
    C = C(\norm{\tv}_{W^{2,\infty}}) > 0
  \] 
  such that 
  \begin{equation}
  \label{eqn:a-err-est}
    \int\limits_{\Omega}E_{rel}(\vrho_{h,\delt,\veps},\bu_{h,\delt,\veps}\vert 1,\tv)(\tau,\cdot)\dbx\lesssim C(\veps + \sqrt{\delt} + \sqrt{h}).
  \end{equation}
\end{theorem}

\begin{proof}
  As in the previous subsection, we divide the proof into two major steps. 

  \textit{Step 1:} Firstly, since $\tv\in W^{2,\infty}(\odom;\R^d)$ is a strong solution of the incompressible system \eqref{eqn:incomp-eul}, it solves 
  \begin{equation}
  \label{eqn:incomp-str-soln}
    \begin{split}
      \Div\tv = 0, \\
      \Dt\tv + (\tv\cdot\grad)\tv + \grad\tpi = 0,
    \end{split}
  \end{equation}
  where $\tpi\in W^{2,\infty}(\odom)$ is the incompressible pressure, which is mean-free, and it solves the elliptic problem 
  \[
    -\Delta_x \tpi = \grad\tv\colon(\grad\tv)^t,\quad \int\limits_\Omega\tpi(t,\cdot)\,\dbx = 0.
  \]
  Next, since the initial data is well-prepared and $\vrho^0_K = (\Pi_{\T}\vrho_\epso)_K$ for any $K\in\T$, observe that from \eqref{eqn:wp-mss-1},
  \[  
    \abs{\vrho^0_K - 1}\leq\frac{1}{\absk}\int_{K}\abs{\vrho_{\epso} - 1}\,\dbx\leq C\veps^2.
  \]
  Consequently,
  \begin{equation}
  \label{eqn:a-err-est-1}
    \norm{\vrho_{h,\delt,\veps}(0,\cdot) - 1}_{L^2(\Omega)}\lesssim\veps^2.
  \end{equation}
  In addition, due to \eqref{eqn:wp-vel-1}, it follows that
  \begin{equation}
  \label{eqn:a-err-est-2}  
    \norm{\bu_{h,\delt,\veps}(0,\cdot) - \tv_0}_{L^2(\Omega;\R^d)}\lesssim\veps + h.
  \end{equation}
  Hence, as a consequence of \eqref{eqn:rel-eng-equiv-incomp-2} and the above estimates, we obtain 
  \begin{equation}
  \label{eqn:a-err-est-ini}
    \int_\Omega E_{rel}(\vrho_{h,\delt,\veps},\bu_{h,\delt,\veps}\vert 1,\tv)(0,\cdot)\,\dbx\lesssim\veps^2 + h^2.
  \end{equation}

  Now, we establish the required relative energy inequality. To begin, we rewrite the relative energy as
  \begin{equation}
  \label{eqn:a-err-est-3}
    \begin{split}
      E_{rel}(\vrho_{h,\delt,\veps},\bu_{h,\delt,\veps}\vert 1,\tv) = &\Bigl(\half\vrho_{h,\delt,\veps}\lvert\bu_{h,\delt,\veps}\rvert^2 + \frac{1}{\veps^2}\Pi(\vrho_{h,\delt,\veps}\vert 1)\Bigr) \\
      &+ \half\vrho_{h,\delt,\veps}\abs{\tv}^2 - \vrho_{h,\delt,\veps}\bu_{\T,\delt, \veps}\cdot\tv.
    \end{split}
  \end{equation}
  Then, for any $\tau\in(0,T)$, the above equation along with the entropy inequality \eqref{eqn:loc-ent-ineq} will yield
  \begin{equation}
  \label{eqn:a-err-est-4}
    \begin{split}
      \biggl\lbrack\int_\Omega E_{rel}(\vrho_{h,\delt,\veps},\bu_{h,\delt,\veps}\vert 1,\tv)(t,\cdot)\,\dbx\biggr\rbrack_{t=0}^{t=\tau} \leq &\biggl\lbrack\int_{\Omega}\half\vrho_{h,\delt,\veps}(t,\cdot)\abs{\tv}^2(t,\cdot)\,\dbx\biggr\rbrack_{t=0}^{t=\tau} \\
      &- \biggl\lbrack\int_{\Omega}\vrho_{h,\delt,\veps}\bu_{h,\delt,\veps}(t,\cdot)\cdot\tv(t,\cdot)\,\dbx\biggr\rbrack_{t=0}^{t=\tau}.
    \end{split} 
  \end{equation}
  Next, we set $\varphi = \half\lvert\tv\rvert^2$ as the test function in \eqref{eqn:cons-mass}, and set $\uphi = \tv$ as the test function in \eqref{eqn:cons-mom}. Proceeding to simplify \eqref{eqn:a-err-est-4}, we obtain
  \begin{multline}
    \left[ \int_{\Omega} E_{rel}(\vrho_{h,\delt,\veps},\bu_{h,\delt,\veps}\vert 1,\tv) \dbx\, \right]_{t=0}^{t=\tau} \leq  \mathcal{S}_{h,\Delta t}^{mass} - \mathcal{S}_{h,\Delta t}^{mom} \\ 
    + \int_0^{\tau} \int_{\Omega} \left[\varrho_{h, \delt, \veps} \paraL{\widetilde{\boldsymbol{v}}-\bu_{h, \delt, \veps}} \cdot \partial_t \widetilde{\boldsymbol{v}}  + \varrho_{h, \delt, \veps} \paraL{\widetilde{\boldsymbol{v}}-\bu_{h, \delt, \veps}} \otimes \bu_{h, \delt, \veps} : \grad \widetilde{\boldsymbol{v}}\right]  \dbx\, \dt,
    \end{multline}
  with 
  \begin{equation}
    \modL{\mathcal{S}_{h,\Delta t}^{mass}} \leq   C_1\paraL{\norm{\frac{1}{2} \modL{\widetilde{\boldsymbol{v}}}^2}_{W^{2,\infty}\paraL{[0,\tau]\times\Omega}}} \paraL{ (1 + \varepsilon)\paraL{\sqrt{h} +\sqrt{\Delta t}}}, 
  \end{equation}
  \begin{equation}
    \modL{\mathcal{S}_{h,\Delta t}^{mom}} \leq  C_2 \paraL{\norm{\widetilde{\boldsymbol{v}}}_{W^{2,\infty}\paraL{[0,\tau]\times \Omega;\mathbb{R}^d}}}  \paraL{ (1 + \varepsilon)\paraL{\sqrt{h} +\sqrt{\Delta t}}}.
  \end{equation}
  Now, since $\tv$ satisfies \eqref{eqn:incomp-str-soln}, we can rewrite the integrand appearing in the above equation as 
  \begin{gather*}
    \vrho_{h,\delt,\veps}(\tv - \bu_{h,\delt,\veps})\cdot\Dt\tv + \vrho_{h,\delt,\veps}(\tv  - \bu_{h,\delt,\veps})\otimes\bu_{h,\delt,\veps}\colon\grad\tv \\
    = -\vrho_{h,\delt,\veps}(\tv - \bu_{h,\delt,\veps})\cdot\grad\tpi - \vrho_{h,\delt,\veps}(\tv  - \bu_{h,\delt,\veps})\otimes(\tv - \bu_{h,\delt,\veps})\colon\grad\tv
  \end{gather*}
  Utilising the above simplification, we get
  \begin{gather}
    \left[ \int_{\Omega} E_{rel}(\vrho_{h,\delt,\veps},\bu_{h,\delt,\veps}\vert 1,\tv) \dbx\, \right]_{t=0}^{t=\tau} \leq  \mathcal{S}_{h,\Delta t}^{mass} - \mathcal{S}_{h,\Delta t}^{mom} - \int_0^\tau\int_{\Omega}\vrho_{h,\delt,\veps}\tv\cdot\grad\tpi\,\dbx\,\dt\nonumber \\
    + \int_0^\tau\int_{\Omega}\vrho_{h,\delt,\veps}\bu_{h,\delt,\veps}\cdot\grad\tpi\,\dbx\,\dt
    - \int_0^\tau\int_{\Omega}\vrho_{h,\delt,\veps}(\tv  - \bu_{h,\delt,\veps})\otimes(\tv - \bu_{h,\delt,\veps})\colon\grad\tv\,\dbx\,\dt. \label{eqn:a-err-est-5}
  \end{gather}
  
  Now, we analyse the remaining terms individually. To begin, since $\Div\tv = 0$, note that $\int_0^\tau\int_\Omega\tv\cdot\grad\tpi\,\dbx\,\dt = \int_0^\tau\int_{\Omega}\Div(\tpi\tv)\,\dbx\,\dt = 0.$ Hence,
  \[
    \int_0^\tau\int_\Omega\vrho_{h,\delt,\veps}\tv\cdot\grad\tpi\,\dbx\,\dt = \int_{0}^\tau\int_{\Omega}(\vrho_{h,\delt,\veps} - 1)\tv\cdot\grad\tpi\,\dbx\,\dt.
  \]
  Invoking Lemma \ref{lem:disc-den-conv}, we immediately obtain
  \begin{equation}
  \label{eqn:a-err-est-6}
    \biggl\lvert\int_0^\tau\int_\Omega\vrho_{h,\delt,\veps}\tv\cdot\grad\tpi\,\dbx\,\dt\biggr\rvert\leq c_1(\norm{\tv}_{W^{2,\infty}((0,\tau)\times\Omega;\R^{d})})\veps
  \end{equation}
  Next, we set $\tpi$ as a test function in \eqref{eqn:cons-mass}. This will yield
  \begin{equation}
  \label{eqn:a-err-est-7}
    \int_0^\tau\int_\Omega\vrho_{h,\delt,\veps}\bu_{h,\delt,\veps}\cdot\grad\tpi\,\dbx\,\dt = \biggl\lbrack\int_\Omega\vrho_{h,\delt,\veps}(t,\cdot)\tpi(t,\cdot)\,\dbx\biggr\rbrack_{t=0}^{t=\tau} - \int_0^\tau\int_\Omega\vrho_{h,\delt,\veps}\Dt\tpi\,\dbx\,\dt - \mathcal{S}^{mass}_{h,\delt}.
  \end{equation}
  Recalling that $\tpi$ is mean-free, we can further simplify the above expression to obtain 
  \begin{equation}
  \label{eqn:a-err-est-8}
    \begin{split}
      \int_0^\tau\int_\Omega\vrho_{h,\delt,\veps}\bu_{h,\delt,\veps}\cdot\grad\tpi\,\dbx\,\dt = &\biggl\lbrack\int_\Omega(\vrho_{h,\delt,\veps}(t,\cdot) - 1)\tpi(t,\cdot)\,\dbx\biggr\rbrack_{t=0}^{t=\tau} \\
      &- \int_0^\tau\int_\Omega(\vrho_{h,\delt,\veps} - 1)\Dt\tpi\,\dbx\,\dt - \mathcal{S}^{mass}_{h,\delt}.
    \end{split}
  \end{equation}
  Therefore, as done earlier, invoking Lemma \ref{lem:disc-den-conv} allows us to obtain
  \begin{equation}
  \label{eqn:a-err-est-9}
    \biggl\lvert\int_0^\tau\int_\Omega\vrho_{h,\delt,\veps}\bu_{h,\delt,\veps}\cdot\grad\tpi\,\dbx\,\dt\biggr\rvert\leq c_2(\norm{\tv}_{W^{2,\infty}((0,\tau)\times\Omega;\R^{d})})\veps + \lvert\mathcal{S}^{mass}_{h,\delt}\rvert,
  \end{equation}
  where $c_2>0$ is a constant which is ultimately dependent on $\norm{\tv}_{W^{2,\infty}}$, as the norms of $\tpi$ depend on the norms of $\tv$. Finally, it straightforwardly follows that 
  \begin{equation}
  \label{eqn:a-err-est-10}
    \begin{split}
      \modL{\int_0^{\tau} \int_{\Omega}  \varrho_{h,\delt,\veps} \paraL{\widetilde{\boldsymbol{v}}-\bu_{h,\delt,\veps}} \otimes\paraL{\bu_{h,\delt,\veps} - \widetilde{\boldsymbol{v}}}  : \grad \widetilde{\boldsymbol{v}}\,\dbx\,\dt} \\
      \lesssim \norm{\grad \widetilde{\boldsymbol{v}}}_{L^{\infty}\paraL{[0,\tau] \times \Omega;\mathbb{R}^{d\times d}}} \int_{0}^{\tau} \int_{\Omega} E(\varrho_{h,\delt,\veps},\bu_{h,\delt,\veps}\vert 1,\widetilde{\boldsymbol{v}})\,\dbx\,\dt. 
    \end{split}
  \end{equation}
  Summarising, from \eqref{eqn:a-err-est-5}-\eqref{eqn:a-err-est-10}, we get
  \begin{multline}
    \left[ \int_{\Omega} E_{rel}(\varrho_{h, \delt,\veps},\bu_{h, \delt,\veps}|1,\tv)\dbx \right]_{t=0}^{t=\tau}  
    \leq c_3\Bigg(\norm{\frac{1}{2} \modL{\tv}^2 }_{W^{2,\infty}\paraL{[0,\tau]\times \Omega}}, \norm{\tv}_{W^{2,\infty}\paraL{[0,\tau]\times\Omega; \mathbb{R}^d}}\Bigg) \\
    \times\paraL{(1+\veps)\sqrt{h}+\sqrt{\Delta t} + \veps} 
    + c_4 \paraL{\norm{\grad \tv}}_{L^{\infty}\paraL{[0,\tau]\times \Omega;\mathbb{R}^{d\times d}}} \int_{0}^{\tau} \int_{\Omega} E_{rel}(\varrho_{h, \delt,\veps},\bu_{h, \delt,\veps}| 1,\tv) \dbx\dt.
  \end{multline}
  Finally, invoking Gronwall's lemma and the estimate on the initial relative energy \eqref{eqn:a-err-est-ini} will yield the desired inequality \eqref{eqn:a-err-est}.
\end{proof}

Due to the equivalence provided by Proposition \ref{prop:rel-eng-equiv-eps}, we have the following corollary.
\begin{corollary}
  Let the conditions stated in Theorem \ref{thm:asymp-err-est} hold. Then, for any $\tau\in(0,T)$,
  \begin{align}
    &\norm{\vrho_{h,\delt,\veps}(\tau,\cdot) - 1}_{L^2(\Omega)}\lesssim \veps(\veps^{\half} + h^{\frac{1}{4}} + \delt^{\frac{1}{4}}), \\
    &\norm{\vrho_{h,\delt,\veps}\bu_{h,\delt,\veps}(\tau,\cdot) - \tv(\tau,\cdot)}_{L^2(\Omega;\R^d)}\lesssim \veps^{\half} + h^{\frac{1}{4}} + \delt^{\frac{1}{4}}
  \end{align}
\end{corollary}

\section{Numerical Results}
\label{sec:num-res}

We illustrate the behaviour of the proposed finite volume method in a series of numerical experiments.  To begin, we implement the following sufficient time-step restriction that ensures the validity of the time-step condition required for the stability of the scheme, namely the condition \textit{(2)} in Theorem \ref{thm:eng-stab}. The proof of this follows analogously as given in \cite{CDV17}, and we refer to it for the same.

\begin{lemma}[Sufficient Time-Step Condition]
  For $\sigma = K\vert L\in\E_{int}$, let $\beta_{\sigma} = \displaystyle\biggl(\max\biggl\lbrace\frac{\abs{\partial K}}{\abs{K}}, \frac{\abs{\partial L}}{\abs{L}}\biggr\rbrace\biggr)^{-1}$. Suppose the time-step $\delt$ satisfies 
  \begin{equation}
    \frac{\delt}{\beta_{\sigma}}\Bigl(\abs{\avg{\bu^n}_{\sigma}\cdot\nuk} + \frac{\abs{\diff{\vrho^{n+1}}_{\sigma}}}{\max{\lbrace\rk^{n+1}, \vrho^{n+1}_L\rbrace}} + \sqrt{\frac{\eta}{\veps^2}\abs{\diff{p^{n+1}}_{\sigma}}}\Bigr)\leq\frac{1}{4}\min\Bigl\lbrace 1, \frac{\min\lbrace\rk^n, \vrho^{n+1}_L\rbrace}{\max{\lbrace\rk^{n+1}, \vrho^{n+1}_L\rbrace}}\Bigr\rbrace.
  \end{equation}
  Then, $\delt$ satisfies the condition \textit{(2)} of Theorem \ref{thm:eng-stab}.
\end{lemma}

Note that the above condition is still implicit in nature. Thus, we implement it in an explicit fashion in our computations.

\subsection*{Stationary Vortex}

We consider the following stationary vortex problem to assess the low Mach behaviour of the present scheme, as the vortex problems act as benchmarks to validate the dissipative properties of the proposed schemes. We consider a square domain $\Omega = \lbrack 0,1\rbrack\times\lbrack 0,1\rbrack$. A radially symmetric vortex is placed at the centre of the domain $(x_1^c,x_2^c) = (0.5,0.5)$, and we let $r = \sqrt{(x_1-x_1^c)^2 + (x_2-x^c_2)^2}$ denote the distance from the centre of the domain. The angular velocity $u_\vth(r)$ is defined as 
\begin{equation*}
    u_\vth(r) = 
    \begin{cases}
        a_1r,&\text{if }r\leq r_1,\\
        a_2 + a_3r,&\text{if }r_1<r\leq r_2,\\
        0,&\text{otherwise}.
    \end{cases}
\end{equation*}
Here, $r_1$ and $r_2$ denote the inner and outer radii, and we choose them as $r_1 = 0.2$ and $r_2 = 0.4$. The other constants are chosen as 
\[
    a_1 = \frac{\overline{a}}{r_1},\quad a_2 = -\frac{\overline{a}r_2}{r_1 - r_2},\quad a_3 = \frac{\overline{a}}{r_1 - r_2},
\]
where $\overline{a} = 0.1$. We begin with an exact steady solution $(\boldsymbol{v}, \pi)$, $\boldsymbol{v} = (v_1, v_2)$, of the incompressible Euler system which is given by
\begin{gather}
    v_1(t, x_1,x_2) = u_\vth(r)\biggr(\frac{x_2 - x_2^c}{r}\biggl),\quad v_2(t,x_1,x_2) = -u_\vth(r)\biggr(\frac{x_1 - x_1^c}{r}\biggl), \\
    \pi(t,x,y) = \int_0^r\frac{u^2_\vth(s)}{s}\mathrm{d}s.
\end{gather}
To obtain the initial data for the compressible system, since the vortex is stationary and has zero radial velocity, it satisfies the balance 
\[
  \frac{1}{\veps^2}\partial_{r}p(\vrho) = \frac{1}{r}\vrho u_{\vth}(r).
\]
Integrating the above equation, we obtain the following initial data for the compressible Euler system:  
\begin{gather}
\label{eqn:vort-ini}
    \vrho(0,x_1,x_2) =  \Bigl(1 + \frac{\gamma\veps^2}{\gamma-1}\pi(0,x,y)\Bigr)^{\frac{1}{\gamma - 1}}, \\
    u_1(0,x_1,x_2) = u_\vth(r)\Bigl(\frac{x_2 - x_2^c}{r}\Bigr),\quad u_2(0,x_1,x_2) = -u_\vth(r)\Bigl(\frac{x_1 - x_1^c}{r}\Bigr). \nonumber
\end{gather}
We set $\gamma = 2$ and the final time is $T = 0.1$. 

To begin, we wish to assess the behaviour of the scheme in the limit $\veps\to 0$. Since $\vrho\to 1$ as $\veps\to 0$, we first present the plots of the deviation $\vrho - 1$ for $\veps = 10^{-i},\ i = 0,\dots,3$, in Figure \ref{fig:den}. As we can see, the overall shape of the vortex is maintained even for smaller values of $\veps$. Also, we plot the relative or the flow Mach number $M = \sqrt{(u_1^2 + u_2^2)/p^\prime(\vrho)}$ at the final time and present them in Figure \ref{fig:rel_mach}. The profiles are nearly identical, leading us to conclude that the dissipation of the scheme is independent of the Mach number. To further strengthen this claim, we compute the relative kinetic energy over time and present it in Figure \ref{fig:rel-ke}. The lines are clearly overlapping, which further corroborates our claim of the dissipation being independent of the Mach number.

\begin{figure}[htpb]
    \centering
    \includegraphics[height = 0.2\textheight]{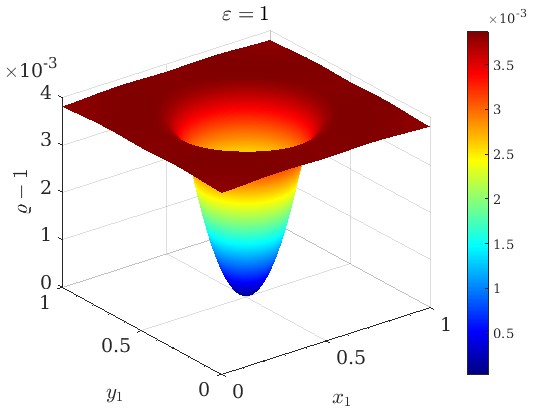}
    \includegraphics[height = 0.2\textheight]{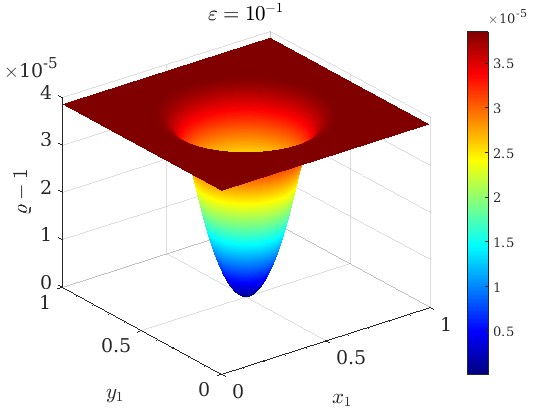}
    \includegraphics[height = 0.2\textheight]{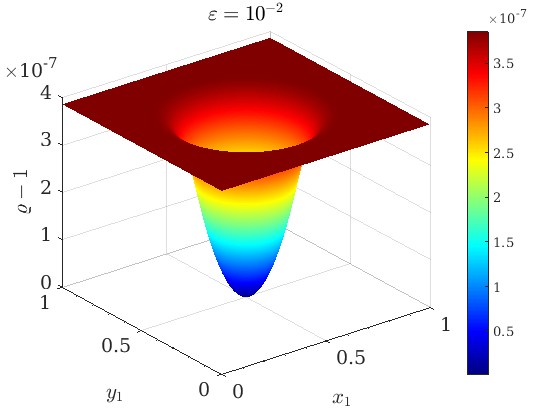}
    \includegraphics[height = 0.2\textheight]{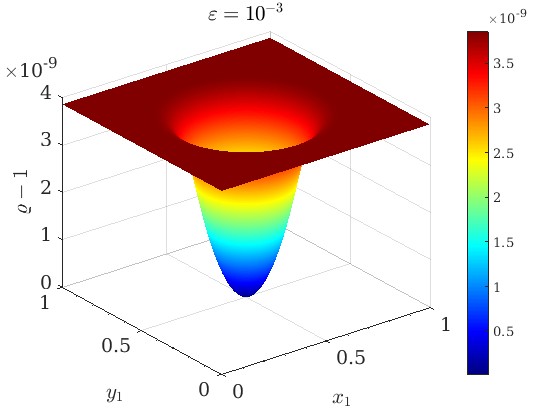}
    \caption{The deviation of the density $\vrho$ from 1 at the final time $T = 0.1$ for different values of $\veps$ on a $512\times 512$ grid.}
    \label{fig:den}
\end{figure}

\begin{figure}[htpb]
    \centering
    \includegraphics[height = 0.2\textheight]{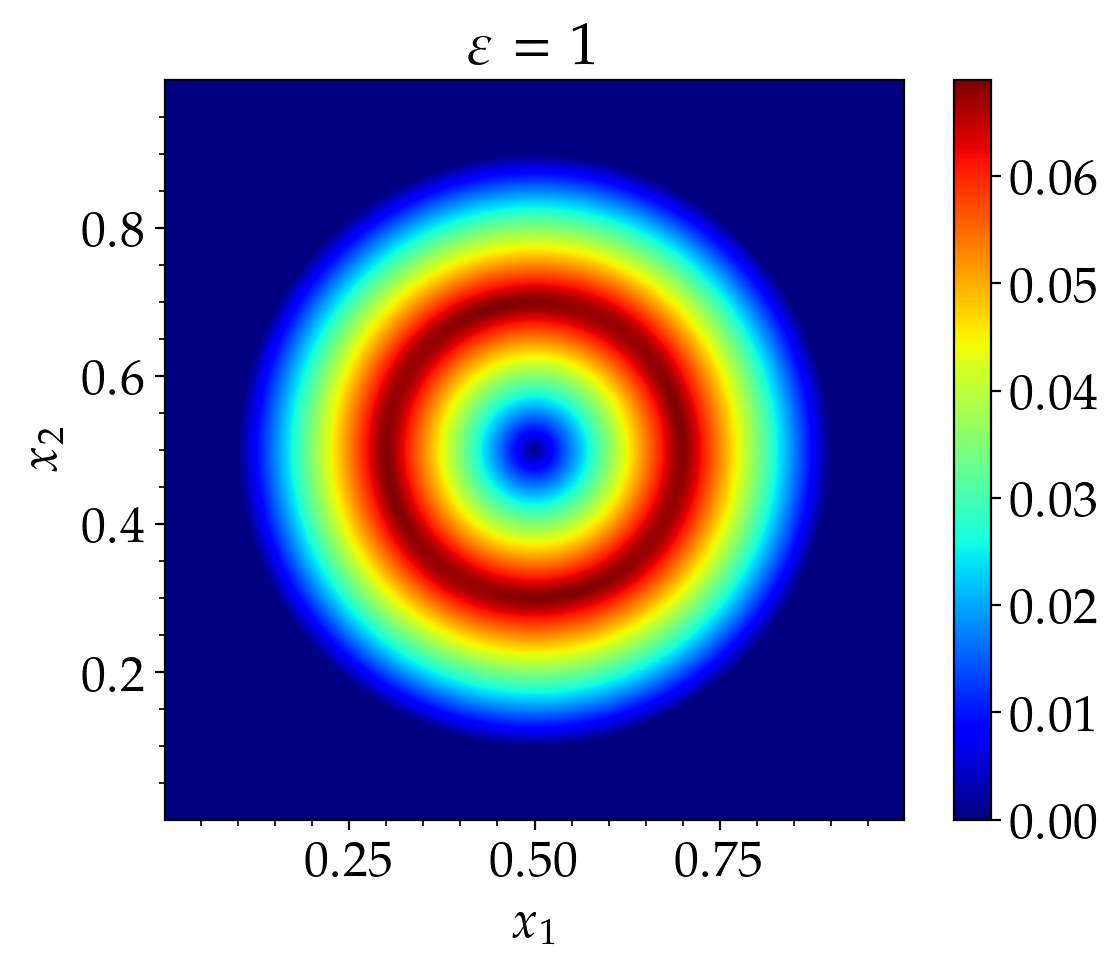}
    \includegraphics[height = 0.2\textheight]{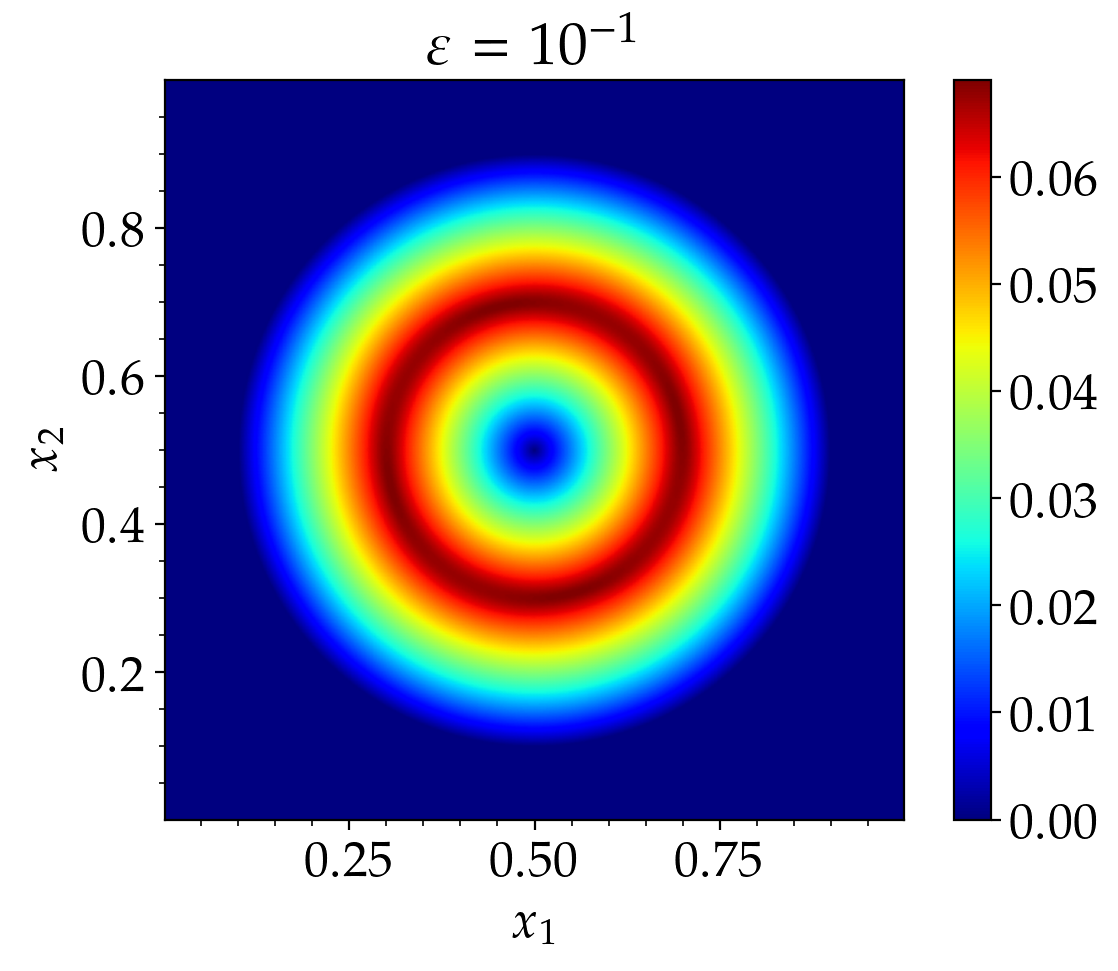}
    \includegraphics[height = 0.2\textheight]{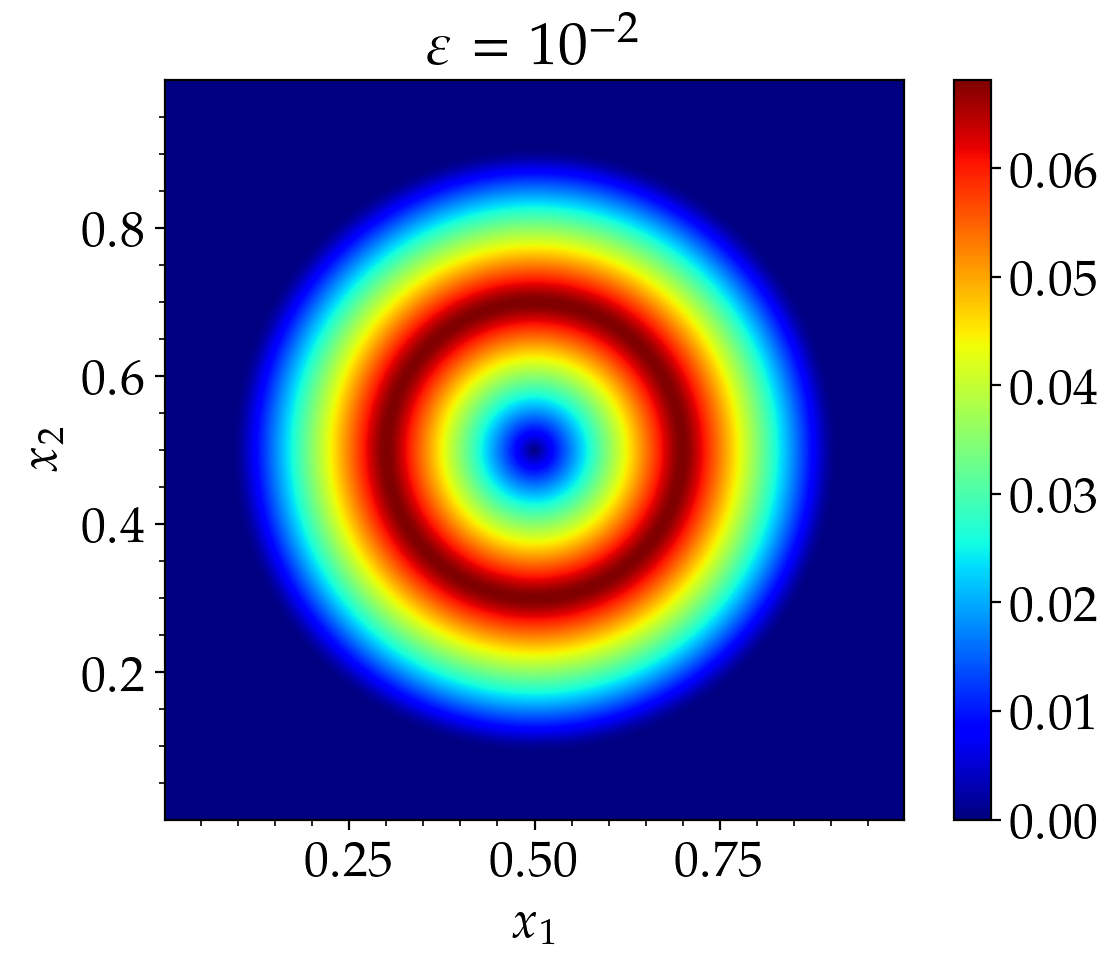}
    \includegraphics[height = 0.2\textheight]{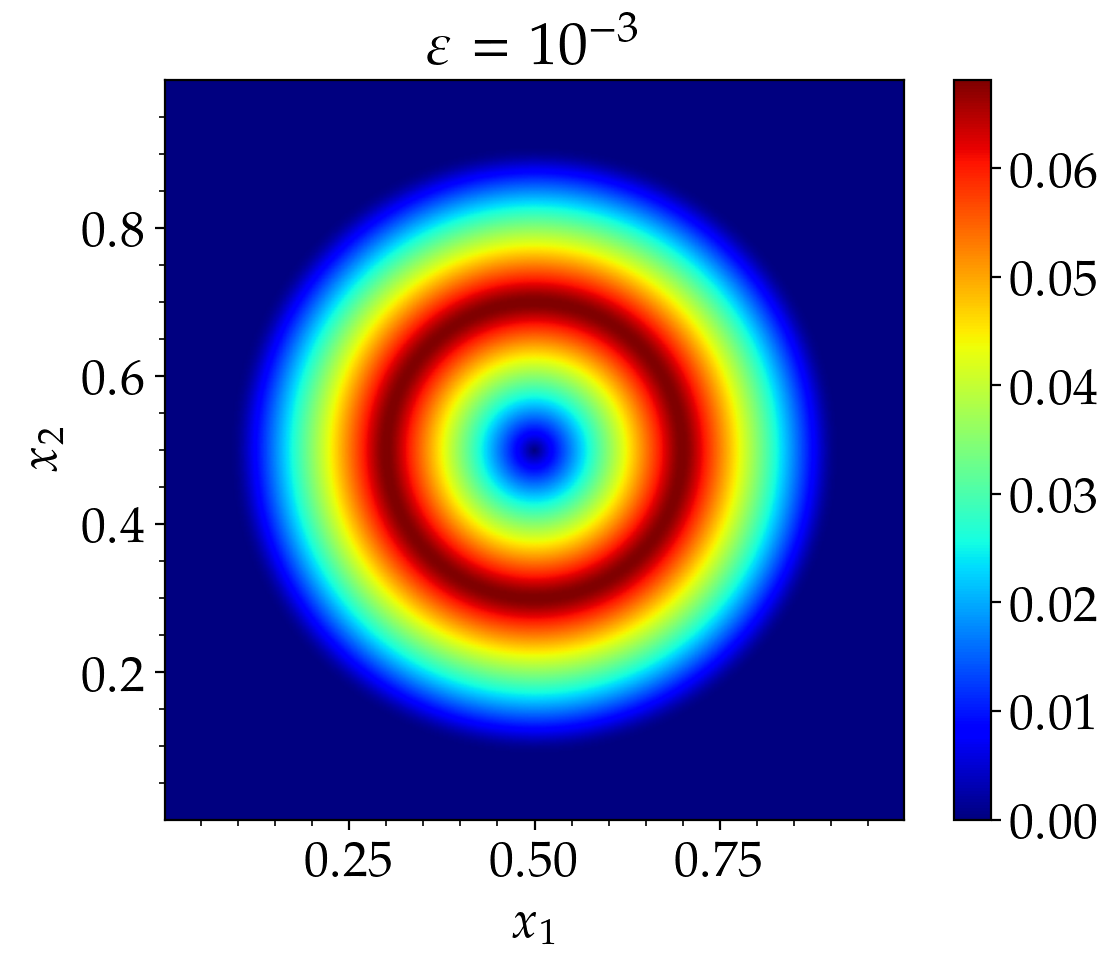}
    \caption{The flow Mach number at the final time $T = 0.1$ for different values of $\veps$ on a $512\times 512$ grid.}
    \label{fig:rel_mach}
\end{figure}

\begin{figure}
    \centering
    \includegraphics[height = 0.3\textheight]{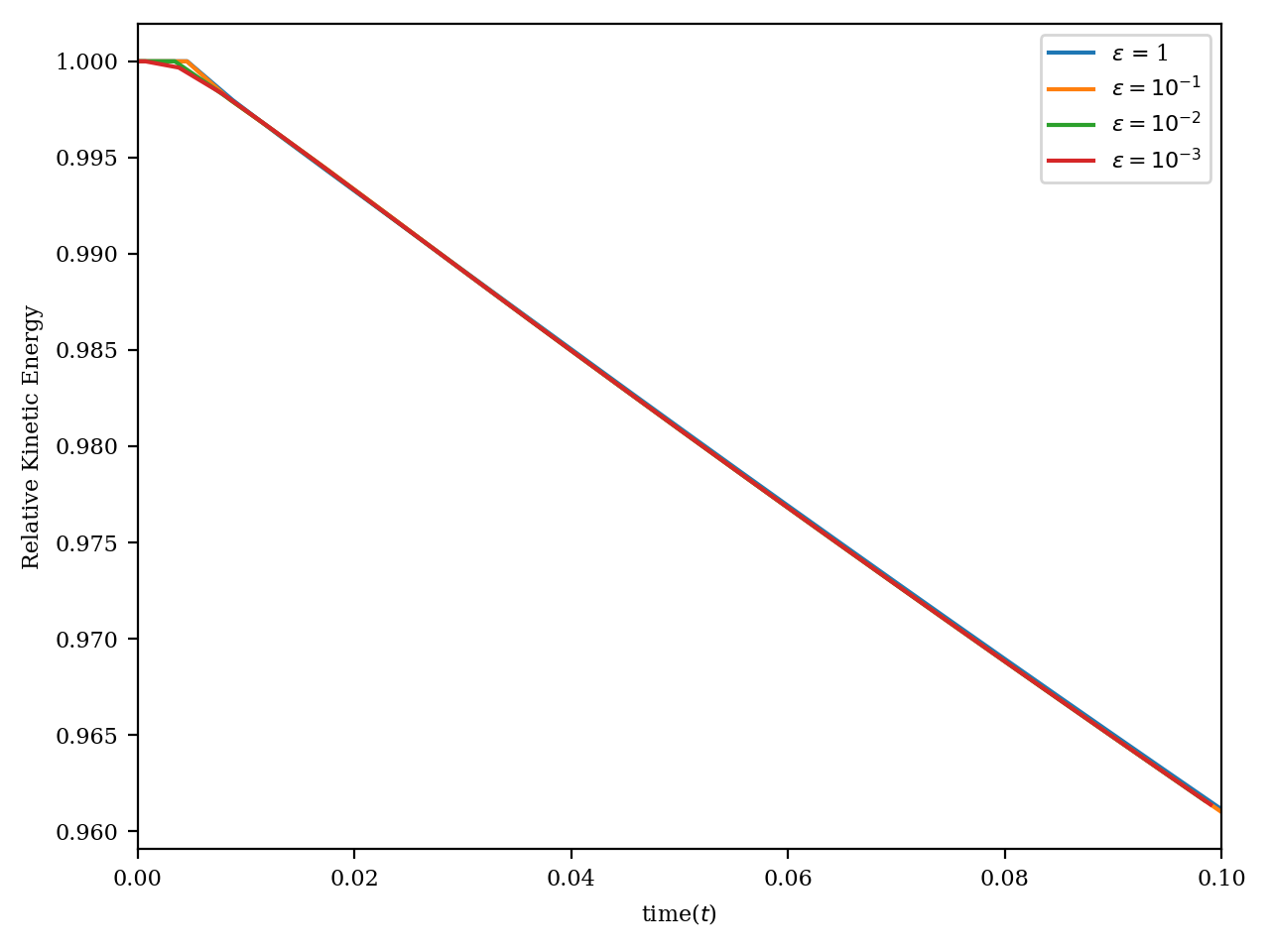}
    \caption{Relative kinetic energy over time for different values of $\veps$.}
    \label{fig:rel-ke}
\end{figure}

Now, our aim is to verify the estimates obtained in Theorem \ref{thm:asymp-err-est}. To do this, we present the experimental order of convergence for different values of $\veps$ in Table \ref{tab:cg-rate}. We compute the solutions $U_{h} = [\vrho_{h}, m_{1, h}, m_{2, h}]$ on successive grids of size $h = 2^{-j}, j = 3,\dots, 8$. Further, we compute a reference solution $U_{ref}$ on a $512\times 512$ grid, and compute the errors of the numerical solutions with respect to the reference solution in the $L^2$ norms. As we can observe, we achieve uniform first-order convergence for all the variables irrespective of the chosen value of $\veps$. Clearly, the obtained rate is better compared to the one claimed in Theorem \ref{thm:err-est}, meaning the theoretically obtained rates are somewhat suboptimal. 

\begin{center}
\begin{table}[htpb]
\begin{tabular}{|c|c|c|c|c|c|c|c|}
    \hline
    $\veps$ & $h$ & $\norm{\vrho_{h} - \vrho_{ref}}_{L^2}$ & EOC & $\norm{m_{1,h} - m_{1,ref}}_{L^2}$ & EOC & $\norm{m_{2,h} - m_{2,ref}}_{L^2}$ & EOC \\
    \hline
            &  1/8   & 7.811e-4 &  -    & 2.132e-2 & -     & 2.111e-2 & -      \\
            &  1/16  & 7.421e-4 & 0.073 & 1.799e-2 & 0.245 & 1.693e-2 & 0.318  \\
    1       &  1/32  & 1.062e-4 & 2.803 & 1.836e-3 & 3.292 & 1.829e-3 & 3.210  \\
            &  1/64  & 5.928e-5 & 0.842 & 1.038e-3 & 0.823 & 1.022e-3 & 0.839  \\
            &  1/128 & 2.924e-5 & 1.019 & 5.129e-4 & 1.017 & 5.121e-4 & 0.996  \\
            &  1/256 & 1.121e-5 & 1.382 & 1.928e-4 & 1.411 & 1.924e-4 & 1.412  \\
    \hline
            &  1/8   & 6.826e-6 &  -    & 2.921e-3 & -     & 2.911e-3 & -      \\
            &  1/16  & 2.301e-6 & 1.568 & 2.747e-3 & 0.088 & 2.643e-3 & 0.139  \\
  $10^{-1}$ &  1/32  & 8.612e-7 & 1.417 & 1.831e-3 & 0.584 & 1.831e-3 & 0.529  \\
            &  1/64  & 4.589e-7 & 0.907 & 1.035e-3 & 0.822 & 1.022e-3 & 0.841  \\
            &  1/128 & 2.105e-7 & 1.124 & 5.121e-4 & 1.015 & 5.120e-4 & 0.997  \\
            &  1/256 & 7.375e-8 & 1.513 & 1.925e-4 & 1.411 & 1.921e-4 & 1.414  \\
    \hline
            &  1/8   & 8.836e-8  &  -    & 2.787e-3 & -     & 2.787e-3 & -      \\
            &  1/16  & 2.562e-8  & 1.786 & 2.554e-3 & 0.125 & 2.554e-3 & 0.125  \\
  $10^{-2}$ &  1/32  & 8.759e-9  & 1.548 & 1.922e-3 & 0.410 & 1.901e-3 & 0.425  \\
            &  1/64  & 5.013e-9  & 0.805 & 1.162e-3 & 0.725 & 1.092e-3 & 0.799  \\
            &  1/128 & 2.219e-9  & 1.175 & 5.543e-4 & 1.068 & 5.431e-4 & 1.007  \\
            &  1/256 & 7.978e-10 & 1.476 & 2.138e-4 & 1.374 & 2.140e-4 & 1.343  \\
    \hline
            &  1/8   & 5.639e-10 &  -    & 2.826e-3 & -     & 2.826e-3 & -      \\
            &  1/16  & 2.903e-10 & 0.957 & 2.504e-3 & 0.174 & 2.504e-3 & 0.174  \\
  $10^{-3}$ &  1/32  & 8.456e-11 & 1.779 & 1.938e-3 & 0.369 & 1.938e-3 & 0.369  \\
            &  1/64  & 4.647e-11 & 0.863 & 1.081e-3 & 0.842 & 1.081e-3 & 0.842  \\
            &  1/128 & 2.235e-11 & 1.056 & 5.578e-4 & 0.954 & 5.578e-4 & 0.954  \\
            &  1/256 & 8.028e-12 & 1.477 & 2.150e-4 & 1.375 & 2.150e-4 & 1.375  \\
    \hline
\end{tabular}
\caption{Experimental order of convergence for different values of $\veps$.}
\label{tab:cg-rate}
\end{table}
\end{center}

Now, we wish to illustrate the convergence of the compressible solution towards the incompressible solution. As done in \cite{FLN+18}, we set $\veps = h = 2^{-j},\ j = 3,\dots, 9$, and we compute the following errors:
\begin{gather*}
   e^{\infty, 1}_{E_{rel}} = \sup_{1\leq n\leq N}\int_\Omega E_{rel}(\vrho^n_{h}, \bu^n_{h}\vert 1, \boldsymbol{v}(t^n,\cdot))\,\dx,\\  
   e^{2,2}_{\vrho} = \norm{\vrho_{h} - 1}_{L^2L^2},\quad e^{\infty,2}_{\vrho} = \sup_{1\leq n\leq N}\norm{\vrho^n_{h} - 1}_{L^2}, \\
   e^{2,2}_{\bu} = \norm{\bu_{h} - \boldsymbol{v}}_{L^2L^2},\quad e^{\infty,2}_{\bu} = \sup_{1\leq n\leq N}\norm{\bu^n_{h} - \boldsymbol{v}(t^n, \cdot)}_{L^2}.
\end{gather*} 
We present the errors in Table \ref{tab:comp-incomp-con-1}. Further, we also compute the errors for the case $\gamma = 1.4$ and present them in Table \ref{tab:comp-incomp-con-2}. As we can observe, the performance of the scheme is independent of the chosen value of $\gamma$. We observe nearly second-order convergence for the densities, while for the relative energy and velocity, we get first-order convergence. Therefore, the convergence rates obtained are far superior compared to the rates obtained theoretically in Theorem \ref{thm:asymp-err-est}. Therefore, the theoretically obtained estimates are still sub-optimal.
\begin{center}
\begin{table}
\begin{tabular}{|c|c|c|c|c|c|c|c|c|c|c|}
    \hline
    $h = \veps$ & $e^{\infty, 1}_{E_{rel}}$ & EOC & $e^{2,2}_\vrho$ & EOC & $e^{\infty,2}_\vrho$ & EOC & $e^{2,2}_\bu$ & EOC & $e^{\infty, 2}_\bu$ & EOC \\
    \hline 
    1/8   & 2.055e-5 & -     & 1.601e-5 & -     & 5.639e-5 & -     & 1.138e-3 & -     & 6.380e-3 & - \\
    1/16  & 5.353e-6 & 1.941 & 3.786e-6 & 2.080 & 1.390e-5 & 2.019 & 6.117e-4 & 0.895 & 3.257e-3 & 0.969 \\
    1/32  & 4.487e-6 & 0.254 & 1.058e-6 & 1.838 & 3.476e-6 & 2.000 & 4.956e-4 & 0.303 & 2.991e-3 & 0.122 \\
    1/64  & 1.694e-6 & 1.405 & 2.630e-7 & 2.009 & 8.690e-7 & 2.000 & 3.010e-4 & 0.719 & 1.839e-3 & 0.701\\
    1/128 & 6.491e-7 & 1.384 & 6.713e-8 & 1.970 & 2.173e-7 & 1.999 & 1.904e-4 & 0.660 & 1.138e-3 & 0.692 \\
    1/256 & 2.363e-7 & 1.457 & 1.698e-8 & 1.982 & 5.450e-8 & 1.995 & 9.859e-5 & 0.949 &
    5.890e-4 & 0.950 \\
    1/512 & 8.467e-8 & 1.481 & 4.286e-9 & 1.986 & 1.414e-8 & 1.946 & 5.198e-5 & 0.923 &
    2.813e-4 & 0.933  \\
    \hline
\end{tabular}
\caption{Convergence of the compressible solution to the incompressible solution ($\gamma = 2$).}
\label{tab:comp-incomp-con-1}
\end{table}
\end{center}

\begin{center}
\begin{table}
\begin{tabular}{|c|c|c|c|c|c|c|c|c|c|c|}
    \hline
    $h = \veps$ & $e^{\infty, 1}_{E_{rel}}$ & EOC & $e^{2,2}_\vrho$ & EOC & $e^{\infty,2}_\vrho$ & EOC & $e^{2,2}_\bu$ & EOC & $e^{\infty, 2}_\bu$ & EOC \\
    \hline 
    1/8   & 2.348e-5 & -     & 1.664e-5 & -     & 5.823e-5 & -     & 1.230e-3 & -     & 6.831e-3 & - \\
    1/16  & 5.763e-6 & 2.026 & 3.890e-6 & 2.096 & 1.404e-5 & 2.051 & 6.447e-4 & 0.932 & 3.384e-3 & 1.013 \\
    1/32  & 4.561e-6 & 0.337 & 1.084e-6 & 1.843 & 3.541e-6 & 1.987 & 5.056e-4 & 0.350 & 3.017e-3 & 0.165 \\
    1/64  & 1.703e-6 & 1.420 & 2.700e-7 & 2.005 & 8.931e-7 & 1.987 & 3.087e-4 & 0.711 & 1.844e-3 & 0.709\\
    1/128 & 6.515e-7 & 1.386 & 6.914e-8 & 1.965 & 2.240e-7 & 1.995 & 2.024e-4 & 0.608 & 1.141e-3 & 0.693 \\
    1/256 & 2.368e-7 & 1.460 & 1.752e-8 & 1.980 & 5.603e-8 & 1.999 & 1.051e-4 & 0.945 &
    5.685e-4 & 0.961 \\
    1/512 & 8.480e-8 & 1.483 & 4.423e-9 & 1.985 & 1.414e-8 & 1.986 & 5.546e-5 & 0.922 &
    2.981e-4 & 0.931  \\
    \hline
\end{tabular}
\caption{Convergence of the compressible solution to the incompressible solution ($\gamma = 1.4$).}
\label{tab:comp-incomp-con-2}
\end{table}
\end{center}

\section{Conclusion}
\label{sec:conc}

We derived uniform error estimates for a semi-implicit-in-time,
energy-stable finite volume scheme approximating the compressible Euler
system parameterised by the Mach number $\veps$. Stability is ensured
via a stabilisation technique that introduces a shifted velocity,
proportional to the stiff pressure gradient, into the convective
fluxes, while consistency across all Mach number regimes enables the
use of the relative energy framework. For a fixed Mach number
$\veps>0$, we proved convergence of the numerical solutions toward a
strong solution of the compressible Euler system with error bounds
uniform in the discretisation parameters. In the low Mach number
regime, the obtained error estimates  yield the convergence toward a strong
solution of the incompressible Euler system as $\veps$ and the
discretisation parameters simultaneously tend to zero. The numerical
experiments are in agreement with the theoretical analysis and
confirm the uniform accuracy and asymptotic-preserving behaviour of
the proposed scheme.

\appendix

\section{Proof of Theorem \ref{thm:eng-stab}}

The proof of this theorem is divided into several steps. We first establish discrete variants of the internal and kinetic energy balances, and proceed to use them to derive the total energy inequality.

\subsection{Internal Energy Balance}

We multiply the mass balance \eqref{eqn:disc-mss} by $P(\rk^{n+1})$. Following the calculations detailed in \cite{AA24}, we obtain the following internal energy balance:
\begin{equation}
\label{eqn:int-eng-bal}
  \frac{1}{\delt}(P(\rk^{n+1}) - P(\rk^n)) + \sum_{\sink}\frac{\abssig}{\absk}H^{n+1}_{\sk} + p^{n+1}_K\sum_{\sink}\frac{\abssig}{\absk}(u^n_\sk - \delta u^n_\sk) + R^{n+1}_{K,\delt} = 0,
\end{equation}
where 
\[
  H^{n+1}_\sk = P(\vrho^{n+1}_{\sigma,up})w^n_\sk - \diff{P(\vrho^{n+1})}_{\sigma}.
\]
The remainder term $R^{n+1}_{K,\delt}$ is given by 
\begin{equation}
\label{eqn:int-eng-rem}
  R^{n+1}_{K,\delt} = \frac{(\varrho_K^{n+1}-\varrho_K^n)^2}{2\Delta t}P^{\prime\prime}(\varrho_{K}^{n+1/2})+\sumintK{}\frac{(\varrho_K^{n+1}-\varrho_L^{n+1})^2}{2}P^{\prime\prime}(\widetilde{\varrho}^{n+1}_{\sigma})(1 -(w_{\sigma,K}^n)_{-}),
\end{equation}
where $\varrho_K^{n+\frac{1}{2}}\in\lbrack\min\{\varrho_K^n,\varrho_K^{n+1}\},\max\{\varrho_K^n,\varrho_K^{n+1}\}\rbrack$ and $\widetilde{\varrho}^{n+1}_\sigma \in \lbrack\min\{\varrho_K^{n+1}, \varrho^{n+1}_L\}, \max\{\varrho_K^{n+1}, \varrho^{n+1}_L\}\rbrack$. The convexity of $P$ ensures that $R^{n+1}_{K,\delt}\geq 0$. 

Next, by the definition of the relative internal energy, cf.\ \eqref{eqn:rel-pot}, note that
\[
  \frac{1}{\delt}(\Pi(\rk^{n+1}\vert 1) - \Pi(\rk^{n}\vert 1)) = \frac{1}{\delt}(P(\rk^{n+1}) - P(\rk^n)) - \frac{P^\prime(1)}{\delt}(\rk^{n+1} - \rk^n).
\]
Using the internal energy balance \eqref{eqn:int-eng-bal} and the mass balance \eqref{eqn:disc-mss} to simplify the above expression will yield the following positive renormalisation identity:
\begin{equation}
\label{eqn:pos-renorm-idt}
  \frac{1}{\delt}(\Pi(\rk^{n+1}\vert 1) - \Pi(\rk^{n}\vert 1)) + \sum_{\sink}\frac{\abssig}{\absk}(H^{n+1}_\sk - P^\prime(1)F^{n+1}_\sk) + p^{n+1}_K\sum_{\sink}\frac{\abssig}{\absk}(u^n_\sk - \delta u^n_\sk) + R^{n+1}_{K,\delt} = 0.
\end{equation}

\subsection{Kinetic Energy Balance}

By rewriting the time derivative term in momentum balance \eqref{eqn:disc-mom} as 
$\varrho_K^{n+1}\bu_K^{n+1} - \varrho_K^n\bu_K^n = \varrho_K^{n+1}(\bu_K^{n+1} - \bu_K^n) + \bu_K^n(\varrho_K^{n+1} - \varrho_K^n)$ and using the mass balance \eqref{eqn:mss-bal} to simplify, we obtain the following velocity balance equation:
\begin{equation}
\label{eqn:vel-bal}
    \frac{\varrho_K^{n+1}(\bu_K^{n+1} - \bu_K^n)}{\Delta t} + \sumintK{}(\bu^n_L - \bu^n_K)F^{n+1,-}_{\sigma,K} + \frac{1}{\varepsilon^2}(\gradt p^{n+1})_K - \sumintK{}\diff{\bu^n}_\sigma = 0.
\end{equation}
Taking the dot product of the above velocity balance equation with $\bu_K^n$ and multiplying \eqref{eqn:disc-mss} by $\modL{\bu_K^n}^2/2$, summing the two expressions and simplifying, we obtain the following kinetic energy balance
\begin{equation}
\label{eqn:kin-eng-bal}
  \frac{1}{\Delta t}\biggl(\frac{1}{2}\varrho_K^{n+1}\modL{\bu_K^{n+1}}^2 - \frac{1}{2}\varrho_K^n\modL{\bu_K^n}^2\biggr) + \sumintK{}Q^{n+1}_{\sigma,K} + \frac{1}{\varepsilon^2}(\gradt p^{n+1})_K\cdot\bu_K^n
  = S^{n+1}_{K,\Delta t},
\end{equation}
where $Q^{n+1}_{\sigma,K} = F^{n+1}_{\sigma,K} \dfrac{\modL{\bu^n_{\sigma,up}}^2}{2} - \frac{1}{2}\diff{\modL{\bu^n}^2}_{\sigma}$ and the remainder term $S^{n+1}_{K,\Delta t}$ is given by 
\begin{equation}
    S^{n+1}_{K,\Delta t} = \frac{\varrho_K^{n+1}}{2\Delta t}\modL{\bu_K^{n+1} - \bu_K^n}^2 - \sumintK{}(1 - F^{n+1,-}_{\sigma,K})\frac{\modL{\bu^n_L - \bu_K^n}^2}{2}. 
\end{equation}

Note that we cannot determine the sign of the above remainder term and thus, we proceed to estimate it as done in \cite{AA24}. Using the velocity update \eqref{eqn:vel-bal} to obtain an estimate on the term $\frac{\varrho_K^{n+1}}{2\Delta t}\modL{\bu_K^{n+1} - \bu_K^n}^2$ appearing in the expression of $S^{n+1}_{K,\delt}$ and proceeding to simplify, we get the following estimate on $S^{n+1}_{K,\delt}$:
\begin{equation}
  \begin{split}
    S^{n+1}_{K,\delt} \leq&\half\Bigl(\sum_{\sink}\frac{\abssig}{\absk}\abs{\diff{\bu^n}_{\sigma}}^2F^{n+1,-}_\sk\Bigr)\Bigl(1 + \frac{3\delt}{\rk^{n+1}}\sum_{\sink}\frac{\abssig}{\absk}F^{n+1, -}_\sk\Bigr) \\
    &+ \Bigl(\frac{3\delt}{2\rk^{n+1}}\frac{\abs{\partial K}}{\absk} - \half\Bigr)\Bigl(\sum_{\sink}\frac{\abssig}{\absk}\abs{\diff{\bu^n}_{\sigma}}^2\Bigr) + \frac{3\delt}{2\veps^4\rk^{n+1}}\abs{(\gradt p^{n+1})_K}^2.
  \end{split}
\end{equation}
Under the conditions \textit{(2)} and \textit{(3)} given in Theorem \ref{thm:eng-stab}, observe that we get 
\begin{equation}
\label{eqn:rem-est}
  S^{n+1}_{K,\delt} \leq -\frac{3\beta}{2}\sum_{\sink}\frac{\abssig}{\absk}\abs{\diff{\bu^n}_{\sigma}}^2 + \frac{3\delt}{2\veps^4\rk^{n+1}}\abs{(\gradt p^{n+1})_K}^2.
\end{equation}

\subsection{Total Energy Inequality}

We multiply \eqref{eqn:int-eng-bal} with $\absk/\veps^2$, multiply \eqref{eqn:kin-eng-bal} with $\absk$, add the two expressions, and sum the resulting expression over all $K\in\T$. This causes the convective flux terms to sum to zero. Further, using discrete summation by parts, observe that
\[
  \sum_{K\in\T}\absk\Bigl(p^{n+1}_K\Bigl(\sum_{\sink}\frac{\abssig}{\absk} u^n_\sk\Bigr) + (\gradt p^{n+1})_K\cdot\bu^n_K\Bigr) = 0.
\]
The resulting expression then reads
\begin{gather}
  \sum_{K\in\T}\frac{\absk}{\delt}\Bigl\lbrack\Bigl(\half\rk^{n+1}\abs{\bu^{n+1}_K}^2 + \frac{1}{\veps^2}P(\rk^{n+1})\Bigr) - \Bigl(\half\rk^{n}\abs{\bu^{n}_K}^2 + \frac{1}{\veps^2}P(\rk^{n})\Bigr)\Bigr\rbrack \nonumber \\
  \leq -\frac{1}{\veps^2}\sum_{\sigma\in\E_{int}}\abs{\dsig}\delta\bu^{n+1}_{\sigma}\cdot(\gradd p^{n+1})_\sigma + \sum_{K\in\T}\absk S^{n+1}_{K,\delt} -\frac{1}{\veps^2}\sum_{K\in\T}\absk R^{n+1}_{K,\delt}.
\end{gather}
Next, substituting for the stabilisation term as $\delta\bu^{n+1}_\sigma = \frac{\eta\delt}{\veps^2}(\gradd p^{n+1})_\sigma$, along with the estimate for the remainder term $S^{n+1}_{K,\delt}$ from \eqref{eqn:rem-est} in the above expression. Simplifying to combine the pressure terms and performing straightforward calculations will yield the following inequality:
\begin{gather}
  \sum_{K\in\T}\frac{\absk}{\delt}\Bigl\lbrack\Bigl(\half\rk^{n+1}\abs{\bu^{n+1}_K}^2 + \frac{1}{\veps^2}P(\rk^{n+1})\Bigr) - \Bigl(\half\rk^{n}\abs{\bu^{n}_K}^2 + \frac{1}{\veps^2}P(\rk^{n})\Bigr)\Bigr\rbrack \nonumber \\
  \leq -\frac{\delt}{\veps^4}\sum_{\sigma\in\E_{int}}\frac{\abssig^2}{\abs{\dsig}}\Bigl(\eta - \frac{3d}{2}\Bigl\{\mskip-5mu\Bigl\{ \frac{1}{\vrho^{n+1}} \Bigr\}\mskip-5mu\Bigr\}_{\sigma}\Bigr)\abs{\diff{p^{n+1}}_{\sigma}}^2 - \frac{3\beta}{2}\sum_{K\in\T}\sum_{\sink}\abssig\abs{\diff{\bu^n}_{\sigma}}^2 -\frac{1}{\veps^2}\sum_{K\in\T}\absk R^{n+1}_{K,\delt}. \label{eqn:eng-ineq-1}
\end{gather}
Every term on the right-hand side of the above equation is non-positive, and this yields the desired stability of the scheme \eqref{eqn:loc-eng-ineq}. To derive the local entropy inequality \eqref{eqn:loc-ent-ineq}, we multiply \eqref{eqn:pos-renorm-idt} with $\absk/\veps^2$, multiply \eqref{eqn:kin-eng-bal} with $\absk$, sum the two, and follow the same steps as we have done so far. 

As a consequence of \eqref{eqn:eng-ineq-1}, we also obtain the following global entropy estimate for each $1\leq n\leq N$:
\begin{gather}
  \sumK{}\Bigl(\half\rk^n\abs{\bu^n_K}^2 + \frac{1}{\veps^2}\Pi(\rk^n\vert 1)\Bigr) +\frac{\Delta t^2}{\varepsilon^4}\sum_{m=0}^{n-1}\sum_{\sigma\in\mathcal{E}_{int}}\frac{\abssig^2}{\abs{\dsig}}\Bigl(\eta - \frac{3d}{2}\Bigl\{\mskip-5mu\Bigl\{ \frac{1}{\vrho^{m+1}} \Bigr\}\mskip-5mu\Bigr\}_{\sigma}\Bigr)\abs{\diff{p^{m+1}}_{\sigma}}^2 \nonumber\\
  +\sum_{n = 0}^{m-1}\Delta t\sum_{K\in\mathcal{T}}\sum_{\sigma\in\mathcal{E}(K)}\frac{\modL{\sigma}}{2\varepsilon^2}\diff{\varrho^{m+1}}^2_{\sigma}P^{\prime\prime}(\tilde{\vrho}^{m+1}_\sigma)
  + \frac{3\beta}{2}\sum_{m = 0}^{n-1}\Delta t\sum_{K\in\mathcal{T}}\sum_{\sigma\in\mathcal{E}(K)}\modL{\sigma}\modL{\diff{\bu^m}_{\sigma}}^2 \label{eqn:glob-ent-est}\\
  +\sum_{m=0}^{n-1}\sumK{}\frac{(\varrho_K^{m+1}-\varrho_K^m)^2}{2\varepsilon^2}P^{\prime\prime}(\varrho_{K}^{m+1/2})\leq\sumK{}\Bigl(\half\rk^0\abs{\bu^0_K}^2 + \frac{1}{\veps^2}\Pi(\rk^0\vert 1)\Bigr) . \nonumber
\end{gather}
The initial datum being well-prepared, cf.\ Definition \ref{def:well-prep}, will ensure that the right-hand side of the above inequality remains bounded independent of $\veps$.

\section{Proof of Theorem \ref{thm:cons}}

As a consequence of Hypothesis \ref{hyp:bnd}, we can obtain the following a priori estimates that are uniform with respect to the mesh parameters and $\veps$ from the global entropy estimate \eqref{eqn:glob-ent-est}, for each $1\leq m\leq N$.
\begin{gather}
  \sum_{n=0}^{m-1}\Delta t\sum_{\sigma\in\mathcal{E}_{int}}\modL{\sigma}\diff{\varrho^{n+1}}^2_{\sigma}\lesssim\varepsilon^2, \quad \sum_{n=0}^{m-1}\Delta t \ \norm{\gradd\varrho^{n+1}}^2_{L^2} \lesssim h^{-1}\varepsilon^2, \quad \sum_{n=0}^{m-1}\sum_{K\in\T}\abs{\varrho^{n+1}_K - \varrho^n_K}^2\lesssim\varepsilon^2, \label{APP:den-est}\\
  \sum_{n = 0}^{m-1}\Delta t \sum_{\sigma\in\mathcal{E}_{int}}\modL{\sigma}\modL{\diff{\bu^n}_{\sigma}}^2\lesssim 1, \quad  \sum_{n=0}^{m-1}\Delta t \ \norm{\gradd\bu^n}^2_{L^2} \lesssim h^{-1},  \label{APP:vel-est}\\
  \sum_{n=0}^{m-1}\sum_{\sigma\in\E_{int}}\abs{\dsig}\abs{\delta\bu^{n+1}_{\sigma}}^2\lesssim 1. \label{APP:stab-est}
\end{gather}

For the proof of consistency, we introduce the following reconstruction operator that transforms piecewise constant functions on the primal grid into piecewise constant functions on the dual mesh. 
\begin{definition}[Reconstruction Operator]
\label{def:recon-op}
  Let $q\in\Lt(\Omega)$ be given. We define the reconstruction of $q$ on the dual mesh as $\mathcal{R}_\T q = \sum_{\sigma\in\E}(\mathcal{R}_\T q)_{\sigma}1_{\dsig}$, where
  \[
    (\mathcal{R}_\mathcal{T} q)_\sigma = \begin{cases}
        \mu_\sigma q_K + (1-\mu_\sigma)q_L,&\text{if }\sigma=K\vert L\in\mathcal{E}_{int}, \\
        q_K, &\text{if }\sigma\in\mathcal{E}_{ext}.
    \end{cases}
\]
\end{definition}

For $1\leq p<\infty$, we have the following available estimates for the reconstruction operator:
\begin{align}
  &\norm{\mathcal{R}_\mathcal{T} q}_{L^p}\lesssim\norm{q}_{L^p}, \label{APP:recon-stab}\\
  &\norm{\mathcal{R}_\mathcal{T} q - q}_{L^p}\lesssim h\norm{\gradd q}_{L^p} \label{APP:recon-est}.
\end{align}

\subsection{Consistency of the mass balance}

For any time-dependent piecewise constant function, we denote $z^n_h = \sum_{K\in \T} z^n_K 1_K$.
Let $\varphi\in W^{2,\infty}((0,T)\times \Omega)$, and let $\varphi^{n+1}_K = (\Pi_\T\varphi(t^{n+1},\cdot))_K$ denote its interpolate on the spacetime grid $(\mathcal{T},\Delta t)$. For $\tau\in(0,T)$, let $0\leq m\leq N-1$ be such that $t^m\leq \tau \leq t^{m+1}$. Multiplying \eqref{eqn:disc-mss} by $\Delta t \modL{K}\varphi_K^{n+1}$, summing over all $K \in \mathcal{T}$ and from $n=0$ to $n=m-1$, and repeatedly using discrete summation by parts will yield \eqref{eqn:cons-mass}, where the consistency error $\mathcal{S}^{mass}_{h,\delt} = \sum_{i = 1}^7 R_i$. The exact expressions and the estimates for the remainder terms $R_i$ are as follows:
\begin{gather*}
  \modL{R_1} = \modL{\sum_{n=0}^{m-1}\sum_{K\in\mathcal{T}}\int_{t^n}^{t^{n+1}}\int_K \varrho_K^{n}\biggl(\partial_t\varphi - \frac{\varphi^{n+1}_K - \varphi^n_K}{\Delta t}\biggr)\,\dbx\,\dt } \lesssim \Delta t\norm{\varphi}_{W^{2,\infty}}, \\
  \modL{R_2} = \modL{\sum_{K\in\mathcal{T}}\int_K\varrho^m_K(\varphi(\tau,\cdot)-\varphi^m_K)\,\dbx - \sum_{K\in\T}\int_K\varrho^0_K(\varphi(0,\cdot)-\varphi^0_K)\dbx} \lesssim h\norm{\varphi}_{W^{2,\infty}},
\end{gather*}
where the estimates follow due to standard interpolation inequalities and Hypothesis \ref{hyp:bnd};
\begin{equation*}
  \modL{R_3} = \modL{\sum_{n=0}^{m-1}\int_{t^n}^{t^{n+1}}\int_\Omega \varrho^{n+1}_\mathcal{D}\bu^{n}_\mathcal{D}\cdot\paraL{\grad\varphi - \gradd\varphi^{n+1}}\dbx\dt} \lesssim h\norm{\varphi}_{W^{2,\infty}}, \\
\end{equation*}
where $\vrho^{n+1}_\D = \sum_{\sigma\in\E}\vrho^{n+1}_{\sigma,up}1_{\dsig}$ and $\bu^n_\D = \sum_{\sigma\in\E}\avg{\bu^n}_{\sigma}1_{\dsig}$;

\begin{align*}
  \modL{R_4} &= \modL{\sum_{n=0}^{m-1}\int_{t^n}^{t^{n+1}}\int_\Omega\paraL{\varrho^{n+1}_h\bu^n_h - \varrho^{n+1}_\mathcal{D}\bu^{n}_\mathcal{D}}\cdot\grad\varphi\,\dbx\,\dt } \nonumber \\ 
  &\lesssim \norm{\varphi}_{W^{2,\infty}}\sum_{n=0}^{m-1}\int_{t^n}^{t^{n+1}} \paraL{\norm{\varrho^{n+1}_h - \varrho^{n+1}_\mathcal{D}}_{L^2} + \norm{\bu^{n}_h - \bu^n_\mathcal{D}}_{L^2}} \,\dt \nonumber \\ 
  &\lesssim  \norm{\varphi}_{W^{2,\infty}}\sum_{n=0}^{m-1}\Delta t h\norm{\gradd\varrho^{n+1}}_{L^2} +  \sum_{n=0}^{m-1}\Delta t h\norm{\gradd\bu^n}_{L^2}  \\ 
  &\lesssim (\varepsilon \sqrt{h} + \sqrt{h})\norm{\varphi}_{W^{2,\infty}}
\end{align*}
wherein we have utilised Hypothesis \ref{hyp:bnd}, the estimates on the reconstruction operator \eqref{APP:recon-est}, and the estimates  \eqref{APP:den-est}, \eqref{APP:vel-est};
\begin{align*}
  \modL{R_5} &= \modL{\sum_{n=0}^{m-1} \int_{t_n}^{t_{n+1}} \int_{\Omega} \paraL{\varrho_{h}^{n} - \varrho_{h}^{n+1}}\bu_{h}^{n} \cdot \grad \varphi \,\dbx \,\dt} \nonumber \\ &\lesssim \norm{\grad\varphi}_{L^{\infty}L^{\infty}}\sum_{n=0}^{m-1} \int_{t_n}^{t_{n+1}} \norm{\varrho_{h}^{n} - \varrho_{h}^{n+1}}_{L^2}  \,\dt \lesssim \varepsilon \sqrt{\Delta t}\norm{\varphi}_{W^{2,\infty}}
\end{align*}
as a consequence of Hypothesis \ref{hyp:bnd} and the estimate \eqref{APP:den-est};
\begin{align*}
  \modL{R_6} &= \modL{ \sum_{n=0}^{m-1}\Delta t \sum_{\sigma\in\mathcal{E}_{int}}\modL{D_{\sigma}}\varrho^{n+1}_{\sigma,up}\delta \bu^{n+1}_{\sigma}\cdot\gradd\varphi^{n+1}_\sigma} \lesssim \sqrt{\Delta t}\norm{\varphi}_{W^{2,\infty}}, 
\end{align*}
where the estimate follows using Hypothesis \ref{hyp:bnd} and \eqref{APP:stab-est};
\begin{align*}
   \modL{R_7} &= \modL{-\sum_{n=0}^{m-1}\Delta t \sum_{K\in\mathcal{T}}\sum_{\sigma \in \mathcal{E}(K)}\modL{\sigma}\diff{\varrho^{n+1}}_{\sigma}\varphi^{n+1}_K} = \modL{\sum_{n=0}^{m-1}\Delta t \sum_{\sigma\in\mathcal{E}}\modL{\sigma}\diff{\varrho^{n+1}}_{\sigma}\diff{\varphi^{n+1}}_\sigma } \nonumber \\ 
   &\lesssim \varepsilon \sqrt{h}\norm{\varphi}_{W^{2,\infty}}, 
\end{align*}
and we obtain the above estimate exactly as done for the term $E_4$ on page 336 of \cite{FLM+21a}, and by using the estimate \eqref{APP:den-est}.

Consequently, we obtain $ \modL{\mathcal{S}^{mass}_{h,\Delta t}}\lesssim C_\varphi((1 + \varepsilon)(\sqrt{h} +\sqrt{\Delta t}) + (h + \Delta t))$, where $C_\varphi>0$ is a constant which depends only on $\varphi$. 

\subsection{Consistency of the momentum balance}

With some minor modifications, the consistency of the time-derivative and flux terms of the momentum balance \eqref{eqn:disc-mom} will also follow along in the same fashion as the mass balance. The only additional term is the pressure term, which is handled on a case-by-case basis for fixed $\veps>0$ and variable $\veps$. 

\subsection{Case I - Fixed $\veps>0$}

In this case, we fix $\veps>0$ (for simplicity, setting $\veps = 1$) and obtain the error estimates between the numerical solutions and a strong solution of the compressible Euler system. Denoting $p_{h,\delt} = \sum_{n=0}^{N-1}\sum_{K\in\T}p^n_K1_K1_{[t^n, t^{n+1})}$, $p^m_h = \sum_{K\in\T} p^m_K 1_K$ for any $0\leq m\leq N$, the consistency of the pressure term will involve accounting for the projection error between the continuous and discrete divergence of the test function, along with the error that shifts the time-level of the pressure term from $n+1$ to $n$. For $\uphi\in W^{2,\infty}(\odom;\R^d)$, we get
\begin{align*}
    \sum_{n=0}^{N-1}\delt\sum_{K\in\T}\absk (\gradt p^{n+1})_K\cdot\uphi^{n+1}_K =& 
    - \int_0^T\int_\Omega p_{h,\delt}\Div\uphi\,\dbx\,\dt \\
    &-\underset{Q_1}{\underleftrightarrow{\sum_{n=0}^{N-1}\int_{t^n}^{t^{n+1}}\int_{\Omega} p^{n+1}_{h}(\divt\uphi^{n+1} - \Div\uphi)\,\dbx\,\dt}}, \\
    &-\underset{Q_2}{\underleftrightarrow{\sum_{n=0}^{N-1}\int_{t^n}^{t^{n+1}}\int_\Omega(p^{n+1}_h - p^n_h)\Div\uphi\,\dbx\,\dt}}.
\end{align*}
The residuals $Q_1$ and $Q_2$ can be straightforwardly bounded. Indeed, for $Q_1$, the boundedness of the density along with standard projection estimates will yield $\abs{Q_1}\lesssim h\norm{\uphi}_{W^{2,\infty}}$. To handle $Q_2$, we can apply the mean value theorem to obtain 
\[
    p^{n+1}_K - p^n_K = p^\prime(\xi^{n+\half}_K)(\vrho^{n+1}_K - \vrho^n_K),
\]
where $\xi^{n+\half}_K\in\mathrm{co}[\rk^n, \rk^{n+1}]$. Consequently, using the boundedness of the density and the a priori estimate $\sum_{n=0}^{N-1}\sum_{K\in\T}\abs{\rk^{n+1} - \rk^n}^2\lesssim 1$ (see \Cref{APP:den-est}, we arrive at 
\[
    \abs{Q_2}\lesssim\norm{\uphi}_{W^{2,\infty}}\sum_{n=0}^{N-1}\delt\sum_{K\in\T}\abs{\rk^{n+1} - \rk^n} \lesssim \sqrt{\delt}\norm{\uphi}_{W^{2,\infty}},
\]

Consequently, in this case, we obtain $ \modL{\mathcal{S}^{mom}_{h,\Delta t}}\lesssim C_{\boldsymbol{\varphi}}((\sqrt{h} +\sqrt{\Delta t}) + (h + \Delta t))$ and $C_{\boldsymbol{\varphi}}>0$ is a constant which depends only on $\boldsymbol{\varphi}$.

\subsection{Case II - Variable $\veps$} In this case, $\veps$ is variable and we obtain the error estimates between the numerical solutions and a strong solution of the incompressible Euler system. In the proof of the error estimates, the pressure term plays no role as we are testing the momentum balance with a divergence-free function. Keeping this in mind, it is sufficient to show that we can bound the projection error arising due to the difference between the discrete and continuous divergence of the test function, uniformly in $\veps$. We do not need to shift the time-levels here, since the term as a whole will end up disappearing. With this insight, we proceed to reformulate the pressure consistency in terms of the second-order pressure term. Let $\pi^{n+1}_K = \dfrac{1}{\veps^2}(p^{n+1}_K - m(p^{n+1}_h))$, $\pi_{h,\delt,\veps} = \sum_{n=0}^{N-1}\sum_{K\in\T}\pi^{n+1}_K1_K1_{[t^n, t^{n+1})}$ denote the second order pressure term, where $m(p^{n+1}_h)$ is the average of $p^{n+1}_h$. Invoking the discrete Poincar\'{e} inequality, observe that
\begin{align}
    \norm{p^{n+1}_h - m(p^{n+1}_h)}_{L^2}^2\lesssim\norm{\gradd p^{n+1}}_{L^2}^2\implies \frac{\delt^2}{\veps^4}\norm{p^{n+1}_h - m(p^{n+1}_h)}^2_{L^2}\lesssim\norm{\du^{n+1}}_{L^2}^2.
\end{align}
Consequently, recalling the estimate $\sum_{n=0}^{N-1}\norm{\du^{n+1}}_{L^2}^2\lesssim 1$ (see \Cref{APP:stab-est}), we can deduce that $\norm{\pi_{h,\delt,\veps}}_{L^2 L^2}\lesssim \delt^{-\half}$. Now, for the consistency of the pressure term, note that
\begin{align*}
    \frac{1}{\veps^2}\sum_{n = 0}^{N-1}\delt\sum_{K\in\T}\absk(\gradt p^{n+1})_K\cdot\uphi^{n+1}_K =& \sum_{n=0}^{N-1}\delt\sum_{K\in\T}\absk(\gradt\pi^{n+1})_K\cdot\uphi^{n+1}_K \\
    = &-\int_0^T\int_\Omega\pi_{h,\delt,\veps}\,\Div\uphi\,\dbx\,\dt \\
    &- \sum_{n=0}^{N-1}\int_{t^{n}}^{t^{n+1}}\int_\Omega\pi^{n+1}_{h}(\divt\uphi^{n+1} - \Div\uphi)\,\dbx\,\dt.
\end{align*}
The residual will now be bounded by $\sqrt{h}\norm{\uphi}_{W^{2,\infty}}$, since $\norm{\pi_{h,\delt,\veps}}_{L^2L^2}\lesssim\delt^{-\half}$ and the projection error will be bounded by $h\norm{\uphi}_{W^{2,\infty}}$. Further, if $\uphi$ is divergence-free, the first term ends up vanishing, and hence, there is no need to do a shift in the time-levels as we did in the earlier case.

Consequently, we obtain $ \modL{\mathcal{S}^{mom}_{h,\Delta t}}\lesssim C_{\boldsymbol{\varphi}}((1+\varepsilon)(\sqrt{h} +\sqrt{\Delta t}) + (h + \Delta t))$, where $C_{\boldsymbol{\varphi}}>0$ is a constant which depends only on $\boldsymbol{\varphi}$.

\bibliographystyle{abbrv}
\bibliography{ref}

@book {GR21,
    AUTHOR = {Godlewski, Edwige and Raviart, Pierre-Arnaud},
     TITLE = {Numerical approximation of hyperbolic systems of conservation
              laws},
    SERIES = {Applied Mathematical Sciences},
    VOLUME = {118},
 PUBLISHER = {Springer-Verlag, New York},
      YEAR = {2021},
     PAGES = {xiii+840},
      ISBN = {978-1-0716-1342-9; 978-1-0716-1344-3},
   MRCLASS = {65M06 (35A35 35L65 65-02 76-02 76Mxx 76N10)},
  MRNUMBER = {4331351},
       DOI = {10.1007/978-1-0716-1344-3},
}

@book {Daf00,
    AUTHOR = {Dafermos, Constantine M.},
     TITLE = {Hyperbolic conservation laws in continuum physics},
    SERIES = {Grundlehren der mathematischen Wissenschaften [Fundamental
              Principles of Mathematical Sciences]},
    VOLUME = {325},
 PUBLISHER = {Springer-Verlag, Berlin},
      YEAR = {2000},
     PAGES = {xvi+443},
      ISBN = {3-540-64914-X},
   MRCLASS = {35L65 (35-02 35L67 74A15 74H20 76A02 76L05 76N10)},
  MRNUMBER = {1763936},
MRREVIEWER = {Denis\ Serre},
       DOI = {10.1007/3-540-29089-3\_14},
}

@book {Tor97,
    AUTHOR = {Toro, Eleuterio F.},
     TITLE = {Riemann solvers and numerical methods for fluid dynamics: A practical introduction},      
 PUBLISHER = {Springer-Verlag, Berlin},
      YEAR = {1997},
     PAGES = {xviii+592},
      ISBN = {3-540-61676-4},
   MRCLASS = {76M20 (35L65 65M12 76N15)},
  MRNUMBER = {1474503},
MRREVIEWER = {Randall J. LeVeque},
       DOI = {10.1007/978-3-662-03490-3},
}

@book {FN17,
    AUTHOR = {Feireisl, Eduard and Novotn\'{y}, Anton\'{\i}n},
     TITLE = {Singular limits in thermodynamics of viscous fluids},
    SERIES = {Advances in Mathematical Fluid Mechanics},
 PUBLISHER = {Birkh\"{a}user/Springer, Cham},
      YEAR = {2017},
     PAGES = {xlii+524},
      ISBN = {978-3-319-63780-8; 978-3-319-63781-5},
   MRCLASS = {35-02 (35B25 35Q35 76-02 76N10)},
  MRNUMBER = {3729430},
}

@book {FLM+21a,
    AUTHOR = {Feireisl, Eduard and Luk\'{a}\v{c}ov\'{a}-Medvid'ov\'{a}, M\'{a}ria and Mizerov\'{a},
              Hana and She, Bangwei},
     TITLE = {Numerical analysis of compressible fluid flows},
    SERIES = {MS\&A. Modeling, Simulation and Applications},
    VOLUME = {20},
 PUBLISHER = {Springer, Cham},
      YEAR = {2021},
     PAGES = {lx+483},
      ISBN = {978-3-030-73787-0; 978-3-030-73788-7},
   MRCLASS = {76Nxx (76-02 76M12)},
  MRNUMBER = {4390192},
       DOI = {10.1007/978-3-030-73788-7},
}

@book {Maj84,
    AUTHOR = {Majda, A.},
     TITLE = {Compressible fluid flow and systems of conservation laws in
              several space variables},
    SERIES = {Applied Mathematical Sciences},
    VOLUME = {53},
 PUBLISHER = {Springer-Verlag, New York},
      YEAR = {1984},
     PAGES = {viii+159},
      ISBN = {0-387-96037-6},
   MRCLASS = {35L65 (76L05 76N10)},
  MRNUMBER = {748308},
MRREVIEWER = {Joel\ Smoller},
       DOI = {10.1007/978-1-4612-1116-7},
}

@article {JR06,
    AUTHOR = {Jovanovi\'c, Vladimir and Rohde, Christian},
     TITLE = {Error estimates for finite volume approximations of classical
              solutions for nonlinear systems of hyperbolic balance laws},
   JOURNAL = {SIAM J. Numer. Anal.},
  FJOURNAL = {SIAM Journal on Numerical Analysis},
    VOLUME = {43},
      YEAR = {2006},
    NUMBER = {6},
     PAGES = {2423--2449},
      ISSN = {0036-1429,1095-7170},
   MRCLASS = {65M12 (35L60 35L65 65M06)},
  MRNUMBER = {2206442},
MRREVIEWER = {Doron\ Levy},
       DOI = {10.1137/S0036142903438136},
}

@article {Vil94,
    AUTHOR = {Vila, J.-P.},
     TITLE = {Convergence and error estimates in finite volume schemes for
              general multidimensional scalar conservation laws. {I}.
              {E}xplicit monotone schemes},
   JOURNAL = {RAIRO Mod\'el. Math. Anal. Num\'er.},
  FJOURNAL = {RAIRO Mod\'elisation Math\'ematique et Analyse Num\'erique},
    VOLUME = {28},
      YEAR = {1994},
    NUMBER = {3},
     PAGES = {267--295},
      ISSN = {0764-583X},
   MRCLASS = {65M12 (65M15 65M60)},
  MRNUMBER = {1275345},
MRREVIEWER = {R\'emi\ Vaillancourt},
       DOI = {10.1051/m2an/1994280302671},
}

@article {TT99,
    AUTHOR = {Tadmor, Eitan and Tang, T.},
     TITLE = {Pointwise error estimates for scalar conservation laws with
              piecewise smooth solutions},
   JOURNAL = {SIAM J. Numer. Anal.},
  FJOURNAL = {SIAM Journal on Numerical Analysis},
    VOLUME = {36},
      YEAR = {1999},
    NUMBER = {6},
     PAGES = {1739--1758},
      ISSN = {0036-1429,1095-7170},
   MRCLASS = {35L65 (65M15)},
  MRNUMBER = {1712177},
MRREVIEWER = {Chi-Wang\ Shu},
       DOI = {10.1137/S0036142998333997},
}

@article {TZ97,
    AUTHOR = {Teng, Zhen-Huan and Zhang, Pingwen},
     TITLE = {Optimal {$L^1$}-rate of convergence for the viscosity method
              and monotone scheme to piecewise constant solutions with
              shocks},
   JOURNAL = {SIAM J. Numer. Anal.},
  FJOURNAL = {SIAM Journal on Numerical Analysis},
    VOLUME = {34},
      YEAR = {1997},
    NUMBER = {3},
     PAGES = {959--978},
      ISSN = {0036-1429},
   MRCLASS = {65M25 (35L65 76L05 76M25 76N10)},
  MRNUMBER = {1451109},
MRREVIEWER = {Hisashi\ Okamoto},
       DOI = {10.1137/S0036142995268862},
}

@article {TT97,
    AUTHOR = {Tang, T. and Teng, Zhen-Huan},
     TITLE = {Viscosity methods for piecewise smooth solutions to scalar
              conservation laws},
   JOURNAL = {Math. Comp.},
  FJOURNAL = {Mathematics of Computation},
    VOLUME = {66},
      YEAR = {1997},
    NUMBER = {218},
     PAGES = {495--526},
      ISSN = {0025-5718,1088-6842},
   MRCLASS = {65M12 (35D05 35L65 65M06)},
  MRNUMBER = {1397446},
MRREVIEWER = {Darrell\ L.\ Hicks},
       DOI = {10.1090/S0025-5718-97-00822-3},
}

@article {CCL94,
    AUTHOR = {Cockburn, Bernardo and Coquel, Fr\'ed\'eric and LeFloch,
              Philippe},
     TITLE = {An error estimate for finite volume methods for
              multidimensional conservation laws},
   JOURNAL = {Math. Comp.},
  FJOURNAL = {Mathematics of Computation},
    VOLUME = {63},
      YEAR = {1994},
    NUMBER = {207},
     PAGES = {77--103},
      ISSN = {0025-5718,1088-6842},
   MRCLASS = {65M15 (35L65 65M60)},
  MRNUMBER = {1240657},
MRREVIEWER = {Anders\ Szepessy},
       DOI = {10.2307/2153563},
}

@article {Kuz76,
    AUTHOR = {Kuznecov, N. N.},
     TITLE = {The accuracy of certain approximate methods for the
              computation of weak solutions of a first order quasilinear
              equation},
   JOURNAL = {\v Z. Vy\v cisl. Mat i Mat. Fiz.},
  FJOURNAL = {\v Zurnal Vy\v cislitel\cprime no\u i\ Matematiki i Matemati\v
              cesko\u i\ Fiziki},
    VOLUME = {16},
      YEAR = {1976},
    NUMBER = {6},
     PAGES = {1489--1502, 1627},
      ISSN = {0044-4669},
   MRCLASS = {65M10 (35A35)},
  MRNUMBER = {483509},
}

@article {HLS21,
    AUTHOR = {Herbin, R. and Latch\'{e}, J.-C. and Saleh, K.},
     TITLE = {Low {M}ach number limit of some staggered schemes for
              compressible barotropic flows},
   JOURNAL = {Math. Comp.},
  FJOURNAL = {Mathematics of Computation},
    VOLUME = {90},
      YEAR = {2021},
    NUMBER = {329},
     PAGES = {1039--1087},
      ISSN = {0025-5718},
   MRCLASS = {65M06 (35Q31 65M12 76M12 76N06)},
  MRNUMBER = {4232217},
MRREVIEWER = {M. K. Kadalbajoo},
       DOI = {10.1090/mcom/3604},
}

@article {CDV17,
    AUTHOR = {Couderc, F. and Duran, A. and Vila, J.-P.},
     TITLE = {An explicit asymptotic preserving low {F}roude scheme for the
              multilayer shallow water model with density stratification},
   JOURNAL = {J. Comput. Phys.},
  FJOURNAL = {Journal of Computational Physics},
    VOLUME = {343},
      YEAR = {2017},
     PAGES = {235--270},
      ISSN = {0021-9991},
   MRCLASS = {76B15 (65M08 76B70 76M12)},
  MRNUMBER = {3654059},
MRREVIEWER = {Paul Andrew Martin},
       DOI = {10.1016/j.jcp.2017.04.018},
}

@article {PV16,
    AUTHOR = {Parisot, Martin and Vila, Jean-Paul},
     TITLE = {Centered-potential regularization for the advection upstream
              splitting method},
   JOURNAL = {SIAM J. Numer. Anal.},
  FJOURNAL = {SIAM Journal on Numerical Analysis},
    VOLUME = {54},
      YEAR = {2016},
    NUMBER = {5},
     PAGES = {3083--3104},
      ISSN = {0036-1429},
   MRCLASS = {65M08 (35B40 35L60 76E20 76M12 86A05)},
  MRNUMBER = {3556070},
MRREVIEWER = {Dami\'{a}n P. Ginestar},
       DOI = {10.1137/15M1021817},
}

@article {DT11,
    AUTHOR = {Degond, Pierre and Tang, Min},
     TITLE = {All speed scheme for the low {M}ach number limit of the
              isentropic {E}uler equations},
   JOURNAL = {Commun. Comput. Phys.},
  FJOURNAL = {Communications in Computational Physics},
    VOLUME = {10},
      YEAR = {2011},
    NUMBER = {1},
     PAGES = {1--31},
      ISSN = {1815-2406},
   MRCLASS = {76M20 (65M06 76L05 76N15)},
  MRNUMBER = {2775032},
MRREVIEWER = {Thomas H. Sonar},
       DOI = {10.4208/cicp.210709.210610a},
}

@article{AGK23,
author={Arun, K. R.
and Ghorai, Rahuldev
and Kar, Mainak},
title={An Asymptotic Preserving and Energy Stable Scheme for the Barotropic {E}uler System in the Incompressible Limit},
journal={J. Sci. Comput.},
year={2023},
month={Nov},
day={07},
volume={97},
number={3},
pages={73},
issn={1573-7691},
doi={10.1007/s10915-023-02389-x},
}

@article {Kat72,
    AUTHOR = {Kato, Tosio},
     TITLE = {Nonstationary flows of viscous and ideal fluids in {${\bf R}\sp{3}$}},
   JOURNAL = {J. Functional Analysis},
  FJOURNAL = {},
      YEAR = {1972},
     PAGES = {296--305},
   MRCLASS = {35Q10},
  MRNUMBER = {481652},
MRREVIEWER = {R.\ Temam},
       DOI = {10.1016/0022-1236(72)90003-1},
}

@article {AA24,
    AUTHOR = {Arun, Koottungal Revi and Krishnamurthy, Amogh},
     TITLE = {A semi-implicit finite volume scheme for dissipative
              measure-valued solutions to the barotropic {E}uler system},
   JOURNAL = {ESAIM Math. Model. Numer. Anal.},
  FJOURNAL = {ESAIM. Mathematical Modelling and Numerical Analysis},
    VOLUME = {58},
      YEAR = {2024},
    NUMBER = {1},
     PAGES = {47--77},
      ISSN = {2822-7840,2804-7214},
   MRCLASS = {35L45 (35D99 35L60 35L65 35L67 35R06 65M08)},
  MRNUMBER = {4688523},
       DOI = {10.1051/m2an/2023093},
}

@article {KM82,
    AUTHOR = {Klainerman, Sergiu and Majda, Andrew},
     TITLE = {Compressible and incompressible fluids},
   JOURNAL = {Comm. Pure Appl. Math.},
  FJOURNAL = {Communications on Pure and Applied Mathematics},
    VOLUME = {35},
      YEAR = {1982},
    NUMBER = {5},
     PAGES = {629--651},
      ISSN = {0010-3640,1097-0312},
   MRCLASS = {35Q20 (35L60 76N10)},
  MRNUMBER = {668409},
MRREVIEWER = {Charles\ J.\ Amick},
       DOI = {10.1002/cpa.3160350503},
}

@article {BLM+23,
    AUTHOR = {Basari\'{c}, Danica and
              Luk\'{a}\v{c}ov\'{a}-Medvid'ov\'{a}, M\'{a}ria and
              Mizerov\'{a}, Hana and She, Bangwei and Yuan, Yuhuan},
     TITLE = {Error estimates of a finite volume method for the compressible
              {N}avier-{S}tokes-{F}ourier system},
   JOURNAL = {Math. Comp.},
  FJOURNAL = {Mathematics of Computation},
    VOLUME = {92},
      YEAR = {2023},
    NUMBER = {344},
     PAGES = {2543--2574},
      ISSN = {0025-5718,1088-6842},
   MRCLASS = {65M15 (65M08 76N06)},
  MRNUMBER = {4628760},
       DOI = {10.1090/mcom/3852},
}

@article {Jin99,
    AUTHOR = {Jin, Shi},
     TITLE = {Efficient asymptotic-preserving ({AP}) schemes for some
              multiscale kinetic equations},
   JOURNAL = {SIAM J. Sci. Comput.},
  FJOURNAL = {SIAM Journal on Scientific Computing},
    VOLUME = {21},
      YEAR = {1999},
    NUMBER = {2},
     PAGES = {441--454},
      ISSN = {1064-8275,1095-7197},
   MRCLASS = {65M06 (76M20 76P05 82C40)},
  MRNUMBER = {1718639},
       DOI = {10.1137/S1064827598334599},
}

@article{Jin2012, 
title={{Asymptotic preserving (AP) schemes for multiscale kinetic and hyperbolic equations: a 
review}}, 
journal={Riv. Mat. Univ. Parma}, 
pages = {177-216}, 
year = {2012},
author={Shi Jin}, 
}

@article {BAL+14,
    AUTHOR = {Bispen, Georgij and Arun, K. R. and
              Luk\'{a}\v{c}ov\'{a}-Medvid'ov\'{a}, M\'{a}ria and Noelle,
              Sebastian},
     TITLE = {I{MEX} large time step finite volume methods for low {F}roude
              number shallow water flows},
   JOURNAL = {Commun. Comput. Phys.},
  FJOURNAL = {Communications in Computational Physics},
    VOLUME = {16},
      YEAR = {2014},
    NUMBER = {2},
     PAGES = {307--347},
      ISSN = {1815-2406,1991-7120},
   MRCLASS = {65M08 (76D05 86A05)},
  MRNUMBER = {3206264},
MRREVIEWER = {Raul\ Borsche},
       DOI = {10.4208/cicp.040413.160114a},
}

@article {NBA+14,
    AUTHOR = {Noelle, S. and Bispen, G. and Arun, K. R. and
              Luk\'{a}\v{c}ov\'{a}-Medvid'ov\'{a}, M. and Munz, C.-D.},
     TITLE = {A weakly asymptotic preserving low {M}ach number scheme for
              the {E}uler equations of gas dynamics},
   JOURNAL = {SIAM J. Sci. Comput.},
  FJOURNAL = {SIAM Journal on Scientific Computing},
    VOLUME = {36},
      YEAR = {2014},
    NUMBER = {6},
     PAGES = {B989--B1024},
      ISSN = {1064-8275,1095-7197},
   MRCLASS = {65M08 (35L65 35Q31 76M12 76M45 76N15)},
  MRNUMBER = {3293458},
MRREVIEWER = {Martin\ Vohral\'{\i}k},
       DOI = {10.1137/120895627},
}

@article {FLN+18,
    AUTHOR = {Feireisl, Eduard and Luk\'{a}\v{c}ov\'{a}-Medvid'ov\'{a}, M\'{a}ria
              and Ne\v{c}asov\'{a}, {\v{S}}\'{a}rka and Novotn\'{y}, Anton\'in and
              She, Bangwei},
     TITLE = {Asymptotic preserving error estimates for numerical solutions
              of compressible {N}avier-{S}tokes equations in the low {M}ach
              number regime},
   JOURNAL = {Multiscale Model. Simul.},
  FJOURNAL = {Multiscale Modeling \& Simulation. A SIAM Interdisciplinary
              Journal},
    VOLUME = {16},
      YEAR = {2018},
    NUMBER = {1},
     PAGES = {150--183},
      ISSN = {1540-3459,1540-3467},
   MRCLASS = {65M60 (35Q30 35Q35 65M12 65M15 76M10 76M12 76N10)},
  MRNUMBER = {3749377},
MRREVIEWER = {Angelamaria\ Cardone},
       DOI = {10.1137/16M1094233},
}

@article {FLS23,
    AUTHOR = {Feireisl, Eduard and Luk\'{a}\v{c}ov\'{a}-Medvid'ov\'{a}, M\'{a}ria and
              She, Bangwei},
     TITLE = {Improved error estimates for the finite volume and the {MAC}
              schemes for the compressible {N}avier-{S}tokes system},
   JOURNAL = {Numer. Math.},
  FJOURNAL = {Numerische Mathematik},
    VOLUME = {153},
      YEAR = {2023},
    NUMBER = {2-3},
     PAGES = {493--529},
      ISSN = {0029-599X,0945-3245},
   MRCLASS = {65N12 (35Q30 76M12 76M20)},
  MRNUMBER = {4557978},
       DOI = {10.1007/s00211-023-01346-y},
}

@article {LS23,
    AUTHOR = {Luk\'{a}\v{c}ov\'{a}-Medvid'ov\'{a}, M\'{a}ria and Sch\"omer, Andreas},
     TITLE = {Compressible {N}avier-{S}tokes equations with potential
              temperature transport: stability of the strong solution and
              numerical error estimates},
   JOURNAL = {J. Math. Fluid Mech.},
  FJOURNAL = {Journal of Mathematical Fluid Mechanics},
    VOLUME = {25},
      YEAR = {2023},
    NUMBER = {1},
     PAGES = {Paper No. 1, 38},
      ISSN = {1422-6928,1422-6952},
   MRCLASS = {76D05 (35Q30)},
  MRNUMBER = {4510551},
       DOI = {10.1007/s00021-022-00733-z},
}

@article {LSY22,
    AUTHOR = {Luk\'{a}\v{c}ov\'{a}-Medvid'ov\'{a}, M\'{a}ria and She, Bangwei and
              Yuan, Yuhuan},
     TITLE = {Error estimates of the {G}odunov method for the
              multidimensional compressible {E}uler system},
   JOURNAL = {J. Sci. Comput.},
  FJOURNAL = {Journal of Scientific Computing},
    VOLUME = {91},
      YEAR = {2022},
    NUMBER = {3},
     PAGES = {Paper No. 71, 27},
      ISSN = {0885-7474,1573-7691},
   MRCLASS = {65M08 (65M12 65M15 76Nxx)},
  MRNUMBER = {4414396},
MRREVIEWER = {Yibing\ Chen},
       DOI = {10.1007/s10915-022-01843-6},
}

@misc{AKL24,
      title={Asymptotic preserving finite volume method for the compressible {E}uler equations: analysis via dissipative measure-valued solutions}, 
      author={K. R. Arun and Amogh Krishnamurthy and Mária Lukáčová-Medvid'ová},
      year={2024},
      journal={arXiv:2405.05685},
      url={https://arxiv.org/abs/2405.05685}, 
}

@article{ALR25,
    author =  {Anandan, Megala and Lukáčová-Medvid'ová, Mária and {S.V.} {Raghurama Rao}},
    year = 2025,
    title  = {An asymptotic preserving scheme satisfying entropy stability for the barotropic {E}uler system},
    journal = {SeMA J.},
    fjournal = {SeMA Journal},
    issn  = {2281-7875},
    doi  = {10.1007/s40324-025-00395-7},
}

@article{AL25,
    author =  {Anandan, Megala and Lukáčová-Medvid'ová, Mária},
    title  = {Provably fully discrete energy stable and asymptotic preserving scheme for barotropic {E}uler equations},
    year={2025},
    journal = {arXiv:2511.19679},
    doi  = {10.48550/arXiv.2511.19679},
}

@article{GHM+15,
    author = {Gallouët, Thierry and Herbin, Raphaèle and Maltese, David and Novotny, Antonin},
    title = {Error estimates for a numerical approximation to the compressible barotropic {N}avier–{S}tokes equations},
    JOURNAL = {IMA J. Numer. Anal.},
  FJOURNAL = {IMA Journal of Numerical Analysis},
    volume = {36},
    number = {2},
    pages = {543-592},
    year = {2015},
    issn = {0272-4979},
    doi = {10.1093/imanum/drv028},
}

@article{BEHL25,
    author = {Brunk, Aaron and Egger, Herbert and Habrich, Oliver and Lukáčová-Medvid’ová, Maria},
    title = {Error analysis for a second order approximation of a viscoelastic phase separation model},
    journal = {Numer. Math.},
    fjournal = {Numerische Mathematik},
    volume = {157},
    number = {5},
    pages = {1449-1489},
    year = {2025},
    issn = {0945-3245},
    doi = {10.1007/s00211-025-01488-1},
}

@article{PR05,
    author = {Pareschi, Lorenzo and Russo, Giovanni},
    title = {Implicit–Explicit {Runge–Kutta} Schemes and Applications to Hyperbolic Systems with Relaxation},
    journal = {J. Sci. Comput.},
    fjournal = {Journal of Scientific Computing},
    volume = {25},
    number = {1},
    pages = {129-155},
    year = {2005},
    issn = {1573-7691},
    doi = {10.1007/s10915-004-4636-4},
}

@article{BPR13,
author = {Boscarino, S. and Pareschi, L. and Russo, G.},
title = {Implicit-Explicit {Runge-Kutta} Schemes for Hyperbolic Systems and Kinetic Equations in the Diffusion Limit},
journal = {SIAM J. Sci. Comput.},
fjournal = {SIAM Journal on Scientific Computing},
volume = {35},
number = {1},
pages = {A22-A51},
year = {2013},
doi = {10.1137/110842855},
}

@article{DP13,
author = {Dimarco, Giacomo and Pareschi, Lorenzo},
title = {Asymptotic Preserving Implicit-Explicit {Runge-Kutta} Methods for Nonlinear Kinetic Equations},
journal = {SIAM J. Numer. Anal.},
fjournal = {SIAM Journal on Numerical Analysis},
volume = {51},
number = {2},
pages = {1064-1087},
year = {2013},
doi = {10.1137/12087606X},
}

@article{ADP20,
author = {Albi, Giacomo and Dimarco, Giacomo and Pareschi, Lorenzo},
title = {Implicit-Explicit Multistep Methods for Hyperbolic Systems With Multiscale Relaxation},
journal = {SIAM J. Sci. Comput.},
fjournal = {SIAM Journal on Scientific Computing},
volume = {42},
number = {4},
pages = {A2402-A2435},
year = {2020},
doi = {10.1137/19M1303290},
}

@article{BQRX22,
author = {Boscarino, Sebastiano and Qiu, Jingmei and Russo, Giovanni and Xiong, Tao},
title = {High Order Semi-implicit WENO Schemes for All-Mach Full Euler System of Gas Dynamics},
journal = {SIAM J. Sci. Comput.},
fjournal = {SIAM Journal on Scientific Computing},
volume = {44},
number = {2},
pages = {B368-B394},
year = {2022},
doi = {10.1137/21M1424433},
}
\end{document}